\documentclass{article}

\usepackage{amsmath}
\usepackage{amssymb}
\usepackage{amsthm}
\usepackage{hyperref}
\usepackage{url}
\usepackage{enumerate}
\usepackage{extarrows}
\usepackage{geometry}

\newtheorem{theorem}{Theorem}[section]
\newtheorem{lemma}[theorem]{Lemma}
\newtheorem{proposition}[theorem]{Proposition}
\newtheorem{corollary}[theorem]{Corollary}

\theoremstyle{definition}

\theoremstyle{remark}
\newtheorem{remark}{Remark}[section]

\allowdisplaybreaks

\begin{document}

\title{Littlewood-Offord problems for the Curie-Weiss models}

\author{Yinshan Chang\thanks{Address: College of Mathematics, Sichuan University, Chengdu 610065, China; Email: ychang@scu.edu.cn}, Xue Peng\thanks{Address: College of Mathematics, Sichuan University, Chengdu 610065, China; Email: pengxuemath@scu.edu.cn}}

\date{}

\maketitle

\begin{abstract}
 In this paper, we consider the Littlewood-Offord problems in one dimension for the Curie-Weiss models. Let
 \[Q_n^{+}:=\sup_{x\in\mathbb{R}}\sup_{v_1,v_2,\ldots,v_n\geq 1}P(\sum_{i=1}^{n}v_i\varepsilon_i\in(x-1,x+1)),\]
 \[Q_n=\sup_{x\in\mathbb{R}}\sup_{|v_1|,|v_2|,\ldots,|v_n|\geq 1}P(\sum_{i=1}^{n}v_i\varepsilon_i\in(x-1,x+1))\]
 where the random variables $(\varepsilon_i)_{1\leq i\leq n}$ are spins in Curie-Weiss models. We calculate the asymptotic properties  of $Q_n^{+}$ and $Q_n$ as $n\to\infty$ and observe the phenomena of phase transitions. Meanwhile,  we also  get  that $Q_n^{+}$ is attained when $v_1=v_2=\cdots=v_n=1$. And  $Q_n$ is attained when one half of $(v_i)_{1\leq i\leq n}$ equals to $1$ and the other half equals to $-1$ when $n$ is even.This is a generalization of classical Littlewood-Offord problems from Rademacher random variables to possibly dependent random variables. In particular, it includes the case of general independent and identically distributed Bernoulli random variables.
\end{abstract}

\section{Introduction}
In this paper, we consider the Littlewood-Offord problems in one dimension for the Curie-Weiss models. To be more precise, we are interested in the asymptotic properties of the following upper bounds
\begin{equation}\label{eq: def Qn}
  Q_n^{+}:=\sup_{x\in\mathbb{R}}\sup_{v_1,v_2,\ldots,v_n\geq 1}P(\sum_{i=1}^{n}v_i\varepsilon_i\in(x-1,x+1)),
\end{equation}
 and
 \begin{equation}\label{eq: def Qn positive}
  Q_n=\sup_{x\in\mathbb{R}}\sup_{|v_1|,|v_2|,\ldots,|v_n|\geq 1}P(\sum_{i=1}^{n}v_i\varepsilon_i\in(x-1,x+1))
 \end{equation}
as $n\to +\infty$, where the random variables $(\varepsilon_i)_{1\leq i\leq n}$ are spins in Curie-Weiss models. Meanwhile, we are also interested in studying that for which  $(v_i)_{1\leq i\leq n}$ and $x$,  $Q^+_n$ and $Q_n$ can be attained. In the following, we call the above two problems about $Q^+_n$ the positive Littlewood-Offord problem and the above two problems about $Q_n$ the Littlewood-Offord problem. And the Curie-Weiss models are spin models from statistic physics. It is a model on complete graphs of $n$ vertices. Each vertex $i$ is associated with a spin $\varepsilon_i\in\{-1,1\}$. We write $\varepsilon=(\varepsilon_1,\varepsilon_2,\ldots,\varepsilon_n)$. For a configuration  $\sigma=(\sigma_1,\sigma_2,\ldots,\sigma_n)\in\{-1,1\}^n$, the probability of realizing a configuration $\sigma$ is
\begin{equation}\label{defn: Curie-Weiss}
P(\varepsilon=\sigma)=\frac{1}{Z_{d,\beta,h}}\exp\left\{\frac{d\beta}{n}(\sum_{j=1}^{n}\sigma_j)^2+h\sum_{j=1}^{n}\sigma_{j}\right\},
\end{equation}
where the positive integer $d$, non-negative real number $\beta$ and real number $h$ are three parameters, and $Z_{d,\beta,h}$ is just the normalizing constant. Usually, $d$ is the dimension, $\beta$ is the inverse temperature and $h$ is the strength of the external field. When $h=0$, we call it the Curie-Weiss model without external field. The Curie-Weiss models without external field have  phase transitions. The critical parameter is
\[\beta_c=\frac{1}{2d}.\]
For $h=0$ and $\beta<\beta_c$, it is called high temperature case. For $h=0$ and $\beta=\beta_c$, it is called critical case. For $h=0$ and $\beta>\beta_c$, it is called low temperature case. When $h\neq 0$, we call it the Curie-Weiss model with non-zero external field. We refer to \cite[Chapter~2]{FriedliVelenikMR3752129} for more details on the Curie-Weiss models. Because $(\varepsilon_i)_{1\leq i\leq n}$ in the Curie-Weiss models are usually not symmetric, the quantities of $Q^+_n$ and $Q_n$ are not the same in general. Thus, in this paper we want to study the  asymptotic properties of $Q^+_n$ and $Q_n$ as $n\to +\infty$ under  high temperature case,  critical case,  low temperature case and with non-zero external field case.

The classical Littlewood-Offord problem concerns the bound $Q_n$ when $(\varepsilon_i)_{1\leq i\leq n}$ are Rademacher random variables. Because the Rademacher random variables are symmetric, $Q^+_n$ and  $Q_n$ are the same.  The first result is due to Littlewood and Offord \cite{LittlewoodOffordMR0009656}. This is the origin of the name of this problem. But their result is not very sharp and contains logarithmic corrections. The first sharp result is given by Erd\H{o}s \cite{ErdosMR0014608} via combinatorial arguments. Erd\H{o}s proved that for Rademacher random variables $(\varepsilon_i)_{1\leq i\leq n}$, it holds that $Q_n=\binom{n}{[n/2]}2^{-n} \sim \sqrt{\frac{2}{\pi}}n^{-\frac{1}{2}}$ as $n\to +\infty$. In the same paper, Erd\H{o}s conjectured that the same sharp upper bound holds if we replace $v_1,v_2,\ldots,v_n$ by $n$ vectors in Hilbert space with norms $||v_i||\geq 1$ and replace $(x-1,x+1)$ by an arbitrary unit open ball $B$ with radius $1$. This long standing conjecture was settled by Kleitman \cite{KleitmanMR0265923}.  Kleitman obtains optimal results for the union of finitely many unit open balls, which is a multi-dimensional analog of Erd\H{o}s's result (\cite[Theorem~3]{ErdosMR0014608}) in dimension one. As far as we know, the paper \cite{TaoVuMR2965282} by Tao and Vu is the most recent development in a similar setup. We refer to the reference there for the series of work in high dimensions. Besides, if $(v_i)_{1\leq i\leq n}$ are different integers, $Q_n$ is of the order $n^{-3/2}$, which  is significantly smaller. We refer the readers to \cite{ErdosMR0174539}, \cite{NguyenMR2891377} and \cite{StanleyMR0578321}  for the related work. Beyond Rademacher series, some people also study $Q_n$ when $(\varepsilon_i)_{1\leq i\leq n}$ are general independent and identically distributed Bernoulli random variables, see \cite{JKMR4201801} and \cite{SinghalMR4440097}. Beyond the independence of $(\varepsilon_i)_{1\leq i\leq n}$, Rao \cite{RaoMR4294326} have studied $(\varepsilon_i)_{1\leq i\leq n}$ driven by Markov chains on a finite state space.

Inspired by the work of Rao \cite{RaoMR4294326}, we continue the study for dependent random variables $(\varepsilon_i)_{1\leq i\leq n}$. We choose to study the Curie-Weiss models. In particular, by \eqref{defn: Curie-Weiss}, when $\beta=0$, $(\varepsilon_i)_{1\leq i\leq n}$ is a sequence of independent and identically distributed Bernoulli random variables. So  by considering Littlewood-Offord problems for the Curie-Weiss models, we generalize the  classical Littlewood-Offord problems for the Rademacher random variables.

Our main results of this paper are the following.

First, we get the asymptotic expansion of the normalizing constant $Z_{d,\beta,h}$. The paper \cite{ShamisZeitouniMR3824953} provides the   asymptotic equivalence of $Z_{d,\beta,h}$. However,  to observe the phase transitions of $Q_n$, we need to compute the second term in the asymptotic expansion of $Z_{d,\beta,h}$ which is not provided in \cite{ShamisZeitouniMR3824953}. Hence, we need to get  the complete asymptotic expansion of $Z_{d,\beta,h}$ by ourselves.

\begin{theorem}\label{thm: normalizing constant expansion}
Let $Z_{d,\beta,h}$ be the normalizing constant of the Curie-Weiss models, i.e.
\begin{equation*}
    Z_{d,\beta,h}:=\sum_{\substack{\sigma=(\sigma_1,\sigma_2,\dots,\sigma_n)\\
    \in\{-1,1\}^n}}\exp\left\{\frac{d\beta}{n}(\sum_{j=1}^{n}\sigma_j)^2+h\sum_{j=1}^{n}\sigma_{j}\right\}.
\end{equation*}
Then, as $n\to+\infty$, the asymptotic expansion of $Z_{d,\beta,h}$ is as follows:
 \begin{enumerate}[(1)]
  \item For $h=0$, $\beta<\beta_c$ and $M\geq 1$, we have that
\[Z_{d,\beta,0}=2^{n}(\sum_{p=0}^{M}e_p^{<}n^{-p}+o(n^{-M})),\]
where $e_0^{<}=(1-2d\beta)^{-\frac{1}{2}}$ and for $p\geq 1$, $e_p^{<}$ is given by \eqref{eq: subcritical cp}.
   \item For $h=0$, $\beta=\beta_c$ and $M\geq 1$, we have that
\[Z_{d,\beta,0}=2^{n}(\sum_{p=0}^{M}e_p^{c}n^{-p/2+1/4}+o(n^{-M/2+1/4})),\]
where $e_0^{c}=\frac{1}{\sqrt{2\pi}}\left(\frac{3}{4}\right)^{\frac{1}{4}}\Gamma(\frac{1}{4})$ and for $p\geq 1$, $e_p^{c}$ is given by \eqref{eq: critical ep}. 
  \item For $h=0, \beta>\beta_c$ and $M\geq 1$, we have that
\[Z_{d,\beta,0}=2^ne^{n\varphi(t_*)}(\sum_{p=0}^{M}e_p^{>}n^{-p}+o(n^{-M})),\]
where $\varphi(t_*)=\ln\cosh t_{*}-\frac{t_{*}^2}{4d\beta}$, $e_0^{>}=2^{\frac{1}{2}}(d\beta)^{-\frac{1}{2}}\left(\frac{1}{2d\beta}-\frac{1}{\cosh^2 t_{*}}\right)^{-\frac{1}{2}}$,
and for $p\geq 1$, $e_p^{>}$ is given by \eqref{eq: hp}. Here, $z_{*}=\frac{t_{*}}{2d\beta}>0$ is the maximal solution of the mean-field equation $\tanh\left(\frac{\beta}{\beta_c}z\right)=z.$
  \item For  $h\neq 0$ and $M\geq 1$, we have that
\[Z_{d,\beta,h}=2^ne^{n\varphi(t_*)}(\sum_{p=0}^{M}e_p^{h\neq 0}n^{-p}+o(n^{-M})),\]
where $\varphi(t_*)=\ln\cosh t_{*}-\frac{(t_{*}-|h|)^2}{4d\beta}$,
$e_0^{h\neq 0}=(2d\beta)^{-\frac{1}{2}}\left(\frac{1}{2d\beta}-\frac{1}{\cosh^2 t_{*}}\right)^{-\frac{1}{2}},$
and for $p\geq 1$, $e_p^{h\neq 0}$ is given by \eqref{eq: h neq ep}. Here, $z_{*}=\frac{t_{*}-|h|}{2d\beta}>0$ is the maximal solution of the mean-field equation
$\tanh\left(\frac{\beta}{\beta_c}z+|h|\right)=z.$
\end{enumerate}
\end{theorem}

Secondly, we get two important results for the positive Littlewood Offord problems $Q^+_n$ for the Curie-Weiss models.
\begin{theorem}\label{thm: positive Littlewood-Offord problems}
Let $n\geq 1$ and  $\varepsilon=(\varepsilon_1, \varepsilon_2,\ldots,\varepsilon_n)$ be the vector of spins in Curie-Weiss models.
Then we have that
\begin{align}\label{eq: Q_n^+ equivalent expression}
Q_n^{+}=&\sup_{x\in\mathbb{R}}P(\varepsilon_1+\varepsilon_2+\cdots+\varepsilon_n\in(x-1,x+1))\notag\\
=&\sup_{x\in\mathbb{R}}P(\varepsilon_1+\varepsilon_2+\cdots+\varepsilon_n=x)\notag\\
=&\max_{k=0,1,\ldots,n}\binom{n}{k}\frac{1}{Z_{d,\beta,h}}\exp\left\{\frac{d\beta}{n}(2k-n)^2+h(2k-n)\right\},
\end{align}
which means $Q^+_n$ is attained for $v_1=v_2=\cdots=v_{n}=1$.
\end{theorem}
By Theorem~\ref{thm: positive Littlewood-Offord problems},  we obtain the asymptotic properties of $Q_n^{+}$ as follows:

\begin{theorem}\label{thm: positive Littlewood-Offord problem asymptotics}
As $n\to+\infty$, the asymptotic equivalence of $Q^+_n$ is as follows:
\begin{enumerate}[(1)]
\item For $h=0$ and $\beta<\beta_{c}$, we have that
\begin{equation*}\label{eq: Q_n^+ subcritical asymptotics}
Q_n^{+}\sim\sqrt{\frac{2(1-2d\beta)}{\pi}}n^{-\frac{1}{2}}.
\end{equation*}

\item For $h=0$ and $\beta=\beta_{c}$, we have that
\begin{equation*}\label{eq: Q_n^+ critical asymptotics}
Q_n^{+}\sim \frac{2}{(\frac{3}{4})^{\frac{1}{4}}\Gamma(\frac{1}{4})}n^{-\frac{3}{4}}.
\end{equation*}

\item For $h=0$ and $\beta>\beta_{c}$, we have that
\begin{equation*}\label{eq: Q_n^+ supercritical asymptotics}
Q_n^{+}\sim\sqrt{\frac{1/(1-z_{*}^2)-2d\beta}{2\pi}}n^{-\frac{1}{2}},
\end{equation*}
where $z_{*}$ is the maximal solution of the mean-field equation
$z=\tanh\left(\frac{\beta}{\beta_{c}}z\right).$

\item For $h\neq 0$, we have that
\begin{equation*}\label{eq: Q_n^+ h neq 0 asymptotics}
Q_n^{+}\sim\sqrt{\frac{2(1/(1-z_{*}^2)-2d\beta)}{\pi}}n^{-\frac{1}{2}},
\end{equation*}
where $z_{*}$ is the maximal solution of the mean-field equation
$z=\tanh\left(\frac{\beta}{\beta_{c}}z+|h|\right).$
\end{enumerate}
\end{theorem}

Next, we get two important results on the Littlewood-Offord problems $Q_n$ for Curie-Weiss models.

\begin{theorem}\label{thm: Littlewood-Offord problems}
Let $n\geq 1$ and $\varepsilon=(\varepsilon_1, \varepsilon_2,\ldots,\varepsilon_n)$ be the vector of spins in Curie-Weiss models.

Then, for even $n$, we have that
\begin{equation}\label{eq: Qn even equality}
Q_n = P\left(\sum_{i=1}^{n/2}\varepsilon_i-\sum_{i=n/2+1}^{n}\varepsilon_i=0\right),
\end{equation}
which means that $Q_n$ is attained for $x=0$, $v_1=v_2=\cdots=v_{n/2}=1$ and $v_{n/2+1}=v_{n/2+2}=\cdots=v_{n}=-1$.

For odd $n$, we have that
\begin{equation}\label{eq: odd n upper lower bounds on Qn}
P\left(\sum_{i=1}^{(n-1)/2}\varepsilon_i-\sum_{i=(n+1)/2}^{n}\varepsilon_i=1\right)\leq Q_{n}\leq Q_{n-1}.
\end{equation}
\end{theorem}

As a consequence of Theorem~\ref{thm: Littlewood-Offord problems}, we have the following asymptotic properties of $Q_n$.
\begin{theorem}\label{thm: Littlewood-Offord problem asymptotics}
As $n\to\infty$, we have that
\[Q_n\sim n^{-\frac{1}{2}}\sqrt{\frac{2}{\pi}}\cosh t_{*},\]
where $z_{*}=\frac{t_{*}-|h|}{2d\beta}\geq 0$ is the maximal solution of the mean-field equation
$\tanh\left(\frac{\beta}{\beta_c}z+|h|\right)=z.$

To be more precise, we give the asymptotic properties of $Q_n$ in four cases as $n\to +\infty$:
\begin{enumerate}[(1)]
  \item For $h=0$ and $\beta<\beta_c$, we have that $t_*=0$ and
  \begin{equation*}
  Q_n=\sqrt{\frac{2}{\pi}}n^{-1/2}+O(n^{-3/2}).
  \end{equation*}
   \item For $h=0$ and $\beta=\beta_c$, we have that $t_*=0$ and
  \begin{equation*}
  Q_n=\sqrt{\frac{2}{\pi}}n^{-1/2}+\frac{2\cdot 3^{\frac{1}{2}}\cdot\pi^{\frac{1}{2}}}{(\Gamma(1/4))^2}n^{-1}+o(n^{-1}).
  \end{equation*}
  \item For $h=0$ and $\beta>\beta_c$, we have that $t_*>0$ and
  \begin{equation*}
   Q_n=\sqrt{\frac{2}{\pi}}\cosh(t_*)n^{-1/2}+O(n^{-3/2}).
  \end{equation*}
  \item For $h\neq 0$, we have that $t_*>0$ and
   \begin{equation*}
   Q_n=\sqrt{\frac{2}{\pi}}\cosh(t_*)n^{-1/2}+O(n^{-3/2}).
  \end{equation*}
\end{enumerate}
\end{theorem}

\begin{remark}
From Theorem~\ref{thm: Littlewood-Offord problem asymptotics}, we observe different first terms of $Q_n$ and different values of $t_*$ with respect to the different cases $h\neq 0$,  $h=0$ with $\beta>\beta_c$ and  $h=0$ with $\beta\leq \beta_c$. And from the second term of $Q_n$, we observe the difference between the critical case $h=0$ with $\beta=\beta_c$ and the non-critical cases.
\end{remark}

\emph{Organization of the paper}: In Section~\ref{sect: asymptotic expansion of Z as a series in n}, we give the  asymptotic expansion of the partition function $Z_{d,\beta,h}$, i.e. we prove Theorem~\ref{thm: normalizing constant expansion}. In Sections~\ref{sect: positive Littlewood-Offord exact}, we study the positive Littlewood-Offord problems for Curie-Weiss models, i.e. we prove Theorem~\ref{thm: positive Littlewood-Offord problems} and Theorem~\ref{thm: positive Littlewood-Offord problem asymptotics}.  In Section~\ref{sect: Littlewood-Offord for iid Bernoulli}, we study the Littlewood-Offord problem for Bernoulli random variables. The result is not trivial and is very important for our study of the Littlewood-Offord problem for Curie-Weiss models in the next section. In Section~\ref{sect: Littlewood-Offord exact and asymptotic}, we study the Littlewood-Offord problems for Curie-Weiss models. First, we prove Theorem~\ref{thm: Littlewood-Offord problems}. Then for even $n$, we give the asymptotic expansion of $Q_n$ in Corollary~\ref{cor: Q asymptotics even}, which means Theorem~\ref{thm: Littlewood-Offord problem asymptotics} holds for even $n$. Then for odd $n$, we give the asymptotic properties of the lower bound of $Q_n$ in Proposition~\ref{prop: Pn expansion odd}, which help us to prove Theorem~\ref{thm: Littlewood-Offord problem asymptotics} for odd $n$. In this way we complete the proof of Theorem~\ref{thm: Littlewood-Offord problem asymptotics}.

\section{Asymptotic expansion of \texorpdfstring{$Z_{d,\beta,h}$}{Z} as a series in \texorpdfstring{$n$}{n}}\label{sect: asymptotic expansion of Z as a series in n}

In this section, we prove Theorem \ref{thm: normalizing constant expansion}. First, we give some useful lemmas.

\begin{lemma}
    Let $Y$ be a standard Gaussian random variable. Then
    \begin{equation}\label{eq: Z and N(0,1)}
        Z_{d,\beta,h}=2^nE\left[\exp\left\{n\ln\cosh\left(\sqrt{\frac{2d\beta}{n}}Y+h\right)\right\}\right].
    \end{equation}
\end{lemma}
\begin{proof}
  By symmetry, $Z_{d,\beta,h}=Z_{d,\beta,|h|}$. Without loss of generality, we  assume $h\geq 0$ in the following calculation. Note that \[Z_{d,\beta,h}=\sum_{\sigma\in\{-1,1\}^n}\exp\left\{\frac{d\beta}{n}(\sum_{j=1}^{n}\sigma_j)^2+h\sum_{j=1}^{n}\sigma_j\right\}.\]
By Hubbard's transformation (i.e. Laplace transform of Gaussian random variables), we have that
\[\exp\left\{\frac{d\beta}{n}(\sum_{j=1}^{n}\sigma_j)^2\right\}=E\left(\exp\left\{\sqrt{\frac{2d\beta}{n}}(\sum_{j=1}^{n}\sigma_j)Y\right\}\right),\]
where $Y$ is a standard Gaussian random variable. Hence, we have that
\begin{align*}
Z_{d,\beta,h}&=\sum_{\sigma\in\{-1,1\}^n}E\left(\exp\left\{(\sqrt{\frac{2d\beta}{n}}Y+h)(\sum_{j=1}^{n}\sigma_j)\right\}\right)\\
&=E\left(\sum_{\sigma\in\{-1,1\}^n}\exp\left\{(\sqrt{\frac{2d\beta}{n}}Y+h)(\sum_{j=1}^{n}\sigma_j)\right\}\right)\\
&=E\left(2^n\cosh^{n}(\sqrt{\frac{2d\beta}{n}}Y+h)\right)\\
&=2^nE\left[\exp\left\{n\ln\cosh\left(\sqrt{\frac{2d\beta}{n}}Y+h\right)\right\}\right].
\end{align*}
\end{proof}

Define
\begin{equation}\label{eq: Z(x)}
Z(x)=E(e^{\ln\cosh(\sqrt{2d\beta}Yx+h)/x^2}).
\end{equation}
Then, by \eqref{eq: Z and N(0,1)} we have that
\begin{equation}\label{eq: Zdbetah via Z(x)}
Z_{d,\beta,h}=2^{n}Z(1/\sqrt{n}).
\end{equation}
In order to find the asymptotic expansion of $Z_{d,\beta,h}$ as $n\to +\infty$, it suffices to find the asymptotic expansion of $Z(x)$  as  $x\to 0^+$. The asymptotic expansion of $Z(x)$ in $x$ can be obtained via Laplace's method. The method  is borrowed from \cite{ErdelyiMR0078494} and \cite{WojdyloMR2219312}.
For $x\neq 0$, by \eqref{eq: Z(x)} we have that
\begin{align}\label{eq: express Z(x) via an integral}
Z(x) 
&= \int_{-\infty}^{+\infty}e^{\ln\cosh(\sqrt{2d\beta}yx+h)/x^2}\frac{1}{\sqrt{2\pi}}e^{-y^2/2}\,\mathrm{d}y\notag\\
&\overset{t=\sqrt{2d\beta}xy+h}{=} \frac{1}{2\sqrt{\pi d\beta}|x|}\int_{-\infty}^{+\infty}e^{\varphi(t)/x^2}\,\mathrm{d}t,
\end{align}
where
\begin{equation}\label{eq: defn varphi t}
\varphi(t)=\ln\cosh t-(t-h)^2/(4d\beta).
\end{equation}

\begin{lemma}\label{lem: the solutions of the mean-field equation}
    Let $z_*$  be the maximal solution of the mean-field equation
$$z=\tanh\left(\frac{\beta}{\beta_{c}}z+|h|\right).$$
\begin{enumerate}[(1)]
\item For $h=0$ and $\beta\leq \beta_c$, we have $z_*=0$. And $0$ is the only solution of the mean-field equation.

\item For $h=0$ and $\beta>\beta_c$, we have $z_*>0$. And $z_*$, $0$ and $-z_*$ are the solutions of the mean-field equation.

\item For $h\neq 0$, we have $z_*>0$. And there exists $\beta_0>\beta_c$ such that for $\beta<\beta_0$, the mean-field equation  has only one solution $z_{*}>0$; for $\beta=\beta_0$, the mean-field equation  has a positive solution $z_{*}$ and a negative solution $z_{2}$; and for $\beta>\beta_0$, the mean-field equation  has a positive solution $z_{*}$ and two different negative solutions $z_{2}$ and $z_{3}$.
\end{enumerate}
\end{lemma}
\begin{proof}
    It is easy to prove by elementary calculation, so we omit it.
\end{proof}

\begin{lemma}\label{lem: the maximum point of varphi}
    Let $\varphi(t)$ be in \eqref{eq: defn varphi t} and $z_*$ be in Lemma \ref{lem: the solutions of the mean-field equation}. Let $t_*=2d\beta z_*+|h|$. Then $t_*$ is the maximum point of $\varphi(t)$.
    \begin{enumerate}[(1)]
\item For $h=0$ and $\beta\leq \beta_c$, we have $t_*=0$. And $0$ is the unique maximum point of $\varphi(t)$.

\item For $h=0$ and $\beta>\beta_c$, we have $t_*>0$. And the  maximum points of $\varphi(t)$ are $t_*$ and $-t_*$.

\item For $h\neq 0$, we have $t_*>0$. And $t_*$ is the unique maximum point of $\varphi(t)$.
\end{enumerate}
\end{lemma}
\begin{proof}
    It is very easy to prove by elementary calculation and Lemma \ref{lem: the solutions of the mean-field equation}, so we omit it.
\end{proof}

In the following, we prove Theorem \ref{thm: normalizing constant expansion} in several subsections with respect to the different phases of the Curie-Weiss models.

\subsection{High temperature case}
In this subsection, we assume $h=0$ and $\beta<\beta_c$. We aim to find the expansion of $Z(x)$ at $x=0$. Hence,  let $x\in (0,0.01)$. Since $h=0$, by \eqref{eq: express Z(x) via an integral} and \eqref{eq: defn varphi t} we have
\begin{equation*}
Z(x)=\frac{1}{\sqrt{\pi d\beta}x}\int_0^{+\infty}e^{\varphi(t)/x^2}\,\mathrm{d}t,
\end{equation*}
where $\varphi(t)=\ln\cosh t-\frac{t^2}{4d\beta}$. By Lemma \ref{lem: the maximum point of varphi}, $0$ is the unique maximum point of $\varphi(t)$. Then, we truncate $Z(x)$ as follows:
\begin{equation}\label{eq: truncation of Z(x) subcritical 01}
Z(x)=Z^{<}(x)+R^{<}(x),
\end{equation}
where
\begin{equation}\label{eq: truncation of Z(x) subcritical 02}
Z^{<}(x)=\frac{1}{\sqrt{\pi d\beta}x}\int_{0}^{-x\ln x}e^{\varphi(t)/x^2}\,\mathrm{d}t \quad \text{and}\quad R^{<}(x)=\frac{1}{\sqrt{\pi d\beta}x}\int_{-x\ln x}^{+\infty}e^{\varphi(t)/x^2}\,\mathrm{d}t.
\end{equation}

First, we estimate $R^{<}(x)$. Note that $\cosh t\leq e^{t^2/2}$ for all $t\in\mathbb{R}$. We have
\[\varphi(t)=\ln\cosh t-\frac{t^2}{4d\beta}\leq \frac{t^2}{2}(1-\frac{1}{2d\beta}).\]
Then, we have that
\begin{align}\label{eq: upper bounds for check R from Z}
R^{<}(x)&\leq\frac{1}{\sqrt{\pi d\beta}x}\int_{-x\ln x}^{+\infty}e^{\frac{1-1/(2d\beta)}{2}(t/x)^2}\,\mathrm{d}t\notag\\
&\xlongequal{t=xs}\frac{1}{\sqrt{\pi d\beta}}\int_{-\ln x}^{+\infty}e^{\frac{1-1/(2d\beta)}{2}s^2}\,\mathrm{d}s\notag\\
&\leq Ce^{-c(\ln x)^2}
\end{align}
for constants $0<c,C<+\infty$. Hence, as $x\to 0$, $R^{<}(x)$ is much smaller than  $x^{M}$  for any $M\geq 0$.

Next, we give the expansion of $Z^{<}(x)$. For sufficiently small $t$, the Taylor expansion of $\varphi(t)$ is
\[\varphi(t)=\sum_{p\geq 2}a_pt^p\]
where
\begin{equation}\label{eq: ap <}
a_2=-(1-2d\beta)/(4d\beta)<0,\quad a_{2p+1}=0,\quad p\geq 1.
\end{equation}
Hence, for any integral $M\geq 2$, we have that
\begin{align}\label{eq: Z< after expansion}
    Z^{<}(x)&=\frac{1}{\sqrt{\pi d\beta}x}\int_{0}^{-x\ln x}e^{\sum_{p\geq 2}a_pt^p/x^2}\,\mathrm{d}t\notag\\
    &\xlongequal{t=xs}\frac{1}{\sqrt{\pi d\beta}}\int_{0}^{-\ln x}e^{a_{2}s^2}e^{\sum_{p=2}^{+\infty}a_{2p}s^{2p}x^{2p-2}}\,\mathrm{d}t\notag\\
    &=\frac{1}{\sqrt{\pi d\beta}}\int_{0}^{-\ln x}e^{a_{2}s^2}e^{\sum_{p=1}^{+\infty}a_{2(p+1)}s^{2(p+1)}x^{2p}}\,\mathrm{d}t\notag\\
    &=\frac{1}{\sqrt{\pi d\beta}}\int_{0}^{-\ln x}e^{a_{2}s^2}(1+\sum_{p=1}^{M}b_p^{<}(s)x^{2p}+r_{M}^{<}(s,x))\,\mathrm{d}t,
\end{align}
where
\begin{equation}\label{eq: expr for bp <}
b_p^{<}(s)=\sum_{k=1}^{p}\sum_{\substack{m_1,m_2,\ldots,m_{k}\geq 1\\ m_1+m_2+\cdots+m_{k}=p}}a_{2(m_1+1)}a_{2(m_2+1)}\cdots a_{2(m_k+1)}s^{2p+2k}/k!
\end{equation}
and there exists $0<C=C(M)<+\infty$ such that for $|s|<-\ln x$, we have that
\begin{equation*}
|r_{M}^{<}(s,x)|\leq Cx^{2M+2}(-\ln x)^{C}=o(x^{2M}).
\end{equation*}
Note that
\begin{align}\label{eq: truncated integral ea2s2s2p}
\int_{0}^{-\ln x}e^{a_2s^2}s^{2p}\,\mathrm{d}s&=\int_{0}^{+\infty}e^{a_2s^2}s^{2p}\,\mathrm{d}s+o(x^{2M})\notag\\
&=\frac{1}{2}\Gamma\left(\frac{2p+1}{2}\right)(-a_2)^{-\frac{2p+1}{2}}+o(x^{2M}).
\end{align}
Combining \eqref{eq: truncation of Z(x) subcritical 01}, \eqref{eq: truncation of Z(x) subcritical 02}, \eqref{eq: upper bounds for check R from Z}, \eqref{eq: ap <}, \eqref{eq: Z< after expansion}, \eqref{eq: expr for bp <} and \eqref{eq: truncated integral ea2s2s2p}, we obtain that for $M\geq 1$,
\begin{equation*}
Z(x)=\sum_{p=0}^{M}e_p^{<}x^{2p}+o(x^{2M}),
\end{equation*}
where $e_0^{<}=(1-2d\beta)^{-\frac{1}{2}}$ and for $p\geq 1$,
\begin{align}\label{eq: subcritical cp}
e_p^{<}&=\sum_{k=1}^{p}(2p+2k-1)!!(2d\beta)^{p+k}(1-2d\beta)^{-\frac{2p+2k+1}{2}}/k!\notag\\
&\quad \times\sum_{\substack{m_1,m_2,\ldots,m_{k}\geq 1\\ m_1+m_2+\cdots+m_{k}=p}}a_{2(m_1+1)}a_{2(m_2+1)}\cdots a_{2(m_k+1)}.
\end{align}
Therefore, by \eqref{eq: Zdbetah via Z(x)}, we obtain
\[Z_{d,\beta,0}=2^{n}(\sum_{p=0}^{M}e_p^{<}n^{-p}+o(n^{-M})),\]
where $e_0^{<}=(1-2d\beta)^{-\frac{1}{2}}$ and for $p\geq 1$, $e_p^{<}$ is given by \eqref{eq: subcritical cp}. Thus,  Theorem \ref{thm: normalizing constant expansion}(1) holds.

\subsection{Critical case}\label{subsubsect: critical}
In this subsection, we assume $h=0$ and $\beta=\beta_c$. We still aim to find the expansion of $Z(x)$ at $x=0$. Hence,  let $x\in (0,0.01)$. Since $h=0$ and $\beta=\beta_c$, by \eqref{eq: express Z(x) via an integral} and \eqref{eq: defn varphi t} we  have
\begin{equation*}
Z(x)=\frac{2}{\sqrt{2\pi}x}\int_0^{+\infty}e^{\varphi(t)/x^2}\,\mathrm{d}t,
\end{equation*}
where $\varphi(t)=\ln\cosh t-\frac{t^2}{2}$. By Lemma \ref{lem: the maximum point of varphi}, $0$ is the unique maximum point of $\varphi(t)$. Then, we truncate $Z(x)$ as follows:
\begin{equation}\label{eq: truncation of Z(x) critical 01}
Z(x)=Z^{c}(x)+R^{c}(x),
\end{equation}
where
\begin{equation}\label{eq: truncation of Z(x) critical 02}
Z^{c}(x)=\frac{2}{\sqrt{2\pi}x}\int_{0}^{\sqrt{-x\ln x}}e^{\varphi(t)/x^2}\,\mathrm{d}t \,\,\text{and} \,\,R^{c}(x)=\frac{2}{\sqrt{2\pi}x}\int_{\sqrt{-x\ln x}}^{+\infty}e^{\varphi(t)/x^2}\,\mathrm{d}t.
\end{equation}

For small enough $t$, the Taylor expansion of $\varphi(t)$ is
\begin{equation}\label{eq: taylor varphi critical 01}
\varphi(t)=\sum_{p=2}^{\infty}a_pt^p
\end{equation}
where
\begin{equation}\label{eq: taylor varphi critical 02}
a_2=0,\,\,a_4=-\frac{1}{12},\,\,a_6=\frac{1}{45},\,\,a_{2p+1}=0,\,\, p\geq 1.
\end{equation}
Thus there exists $\delta>0$ and $\kappa>0$ such that for $|t|\leq\delta$, $\varphi(t)<-\kappa t^{4}$. Since $t^2/2-\ln\cosh t>0$ for $t\neq 0$ and $\lim_{t\to\infty}\frac{t^2/2-\ln\cosh t}{|t|}=+\infty$, for the same $\delta$, there exists $\widetilde{\kappa}>0$ such that for $|t|>\delta$,  $\varphi(t)<-\widetilde{\kappa}|t|$. Therefore, we have
\begin{equation}\label{eq: upper bounds for tilde R from Z}
R^{c}(x)\leq \frac{2}{\sqrt{2\pi}x}\left(\int_{\sqrt{-x\ln x}}^{\delta}e^{-\kappa t^4/x^2}\,\mathrm{d}t+\int_{\delta}^{+\infty}e^{-\widetilde{\kappa}t/x^2}\,\mathrm{d}t\right)\leq Ce^{-c(\ln x)^2}
\end{equation}
for some constants $0<c,C<+\infty$. Here, as $x\to 0$, $R^{c}(x)$ is much smaller than  $x^{M}$  for any $M\geq 1$.

For the main part $Z^{c}(x)$, by \eqref{eq: taylor varphi critical 01} and \eqref{eq: taylor varphi critical 02}  we have that
\begin{align}\label{eq: critical expansion of tilde Z via integral}
Z^{c}(x) &=\frac{2}{\sqrt{2\pi}x}\int_{0}^{\sqrt{-x\ln x}}e^{(\sum_{p=2}^{+\infty}a_{2p}t^{2p})/x^2}\,\mathrm{d}t\notag\\
&\xlongequal{t=s\sqrt{x}} \frac{2}{\sqrt{2\pi x}}\int_{0}^{\sqrt{-\ln x}}e^{-\frac{1}{12}s^{4}}e^{\sum_{p=3}^{+\infty}a_{2p}s^{2p}x^{p-2}}\,\mathrm{d}s\notag\\
&=\frac{2}{\sqrt{2\pi x}}\int_{0}^{\sqrt{-\ln x}}e^{-\frac{1}{12}s^{4}}e^{\sum_{p=1}^{+\infty}a_{2p+4}s^{2p+4}x^{p}}\,\mathrm{d}s\notag\\
&=\frac{2}{\sqrt{2\pi x}} \int_{0}^{\sqrt{-\ln x}}e^{-\frac{1}{12}s^{4}}(1+\sum_{p=1}^{M}b_{p}^{c}(s)x^{p}+r_{M}^{c}(s,x))\,\mathrm{d}s,
\end{align}
where
\begin{equation}\label{eq: expr for bp c}
b_p^{c}(s)=\sum_{k=1}^{p}\sum_{\substack{m_1,m_2,\ldots,m_{k}\geq 1\\ m_1+m_2+\cdots+m_{k}=p}}a_{2m_1+4}a_{2m_2+4}\cdots a_{2m_k+4}s^{2p+4k}/k!
\end{equation}
and there exists $0<C<+\infty$ such that for $|s|<\sqrt{-\ln x}$, we have that
\[|r_{M}^{c}(s,x)|\leq Cx^{M+1}(-\ln x)^{C}=o(x^{M}).\]
Note that for non-negative integer $p$,
\begin{align}\label{eq: truncated e-s4s2p integral}
\int_{0}^{\sqrt{-\ln x}}e^{-\frac{1}{12}s^{4}}s^{2p}\,\mathrm{d}s &=\int_{0}^{+\infty}e^{-\frac{1}{12}s^{4}}s^{2p}\,\mathrm{d}s+o(x^{M})\notag\\
&=\frac{1}{2}\left(\frac{3}{4}\right)^{\frac{1}{4}}12^{\frac{p}{2}}\Gamma(\frac{p}{2}+\frac{1}{4})+o(x^{M}).
\end{align}
Hence, combining \eqref{eq: truncation of Z(x) critical 01}, \eqref{eq: truncation of Z(x) critical 02}, \eqref{eq: upper bounds for tilde R from Z}, \eqref{eq: critical expansion of tilde Z via integral}, \eqref{eq: expr for bp c} and \eqref{eq: truncated e-s4s2p integral}, we obtain that for $M\geq 1$,
\begin{equation*}
Z(x)=\sum_{p=0}^{M}e_p^{c}x^{p-\frac{1}{2}}+o(x^{M-\frac{1}{2}}),
\end{equation*}
where $e_0^{c}=\frac{1}{\sqrt{2\pi}}\left(\frac{3}{4}\right)^{\frac{1}{4}}\Gamma(\frac{1}{4})$ and for $p\geq 1$,
\begin{align}\label{eq: critical ep}
e_p^{c}&=\frac{1}{\sqrt{2\pi}}\left(\frac{3}{4}\right)^{\frac{1}{4}}\sum_{k=1}^{p}12^{\frac{p}{2}+k}\Gamma(\frac{p}{2}+k+\frac{1}{4})/k!\notag\\
&\quad\times\sum_{\substack{m_1,m_2,\ldots,m_{k}\geq 1\\ m_1+m_2+\cdots+m_{k}=p}}a_{2m_1+4}a_{2m_2+4}\cdots a_{2m_k+4}.
\end{align}
Therefore, by \eqref{eq: Zdbetah via Z(x)}, we have that
\[Z_{d,\beta,0}=2^{n}(\sum_{p=0}^{M}e_p^{c}n^{-p/2+1/4}+o(n^{-M/2+1/4})),\]
where $e_0^{c}=\frac{1}{\sqrt{2\pi}}\left(\frac{3}{4}\right)^{\frac{1}{4}}\Gamma(\frac{1}{4})$ and for $p\geq 1$, $e_p^{c}$ is given by \eqref{eq: critical ep}. In particular, $e_1^{c}=\frac{1}{5\sqrt{\pi}}3^{\frac{3}{4}}\Gamma(\frac{3}{4})$. Thus,  Theorem \ref{thm: normalizing constant expansion}(2) holds.

\begin{remark}
    The truncation of $Z(x)$ (i.e. \eqref{eq: truncation of Z(x) critical 01} and \eqref{eq: truncation of Z(x) critical 02}) in the critical case is different from that (i.e. \eqref{eq: truncation of Z(x) subcritical 01} and \eqref{eq: truncation of Z(x) subcritical 02}) in the high temperature case because from \eqref{eq: taylor varphi critical 02} we know that $a_2=0, a_4\neq 0$ in the critical case, while in \eqref{eq: ap <} we see $a_2\neq 0$ in the high temperature case.
\end{remark}

\subsection{Low temperature case}\label{subsection: Z low temperature case}
In this subsection, we assume $h=0$ and $\beta>\beta_c$. Again, we aim to find the expansion of $Z(x)$ at $x=0$. Hence, let $x\in (0,0.01)$. Since $h=0$, by \eqref{eq: express Z(x) via an integral} and \eqref{eq: defn varphi t} we have
\begin{equation*}
Z(x)=\frac{1}{\sqrt{\pi d\beta}x}\int_0^{+\infty}e^{\varphi(t)/x^2}\,\mathrm{d}t,
\end{equation*}
where $\varphi(t)=\ln\cosh t-\frac{t^2}{4d\beta}$. By Lemma \ref{lem: the maximum point of varphi}, the maximum points of $\varphi(t)$ are $t_*$ and $-t_*$. Then, we truncate $Z(x)$ as follows:
\begin{equation}\label{eq: truncation Z(x) low temperature 01}
Z(x)=Z^{>}(x)+R^{>}(x),
\end{equation}
where
\begin{align}\label{eq: truncation Z(x) low temperature 02}
Z^{>}(x)&=\frac{1}{\sqrt{\pi d\beta}x}\int_{t_{*}+x\ln x}^{t_{*}-x\ln x}e^{\frac{\varphi(t)}{x^2}}\,\mathrm{d}t,\\
R^{>}(x)&=\frac{1}{\sqrt{\pi d\beta}x}\left(\int_{0}^{t_{*}+x\ln x}e^{\frac{\varphi(t)}{x^2}}\,\mathrm{d}t+\int_{t_{*}-x\ln x}^{+\infty}e^{\frac{\varphi(t)}{x^2}}\,\mathrm{d}t\right).\notag
\end{align}

For $t$ sufficiently close to $t_{*}$, the Taylor expansion of $\varphi(t)$ is
\begin{equation}\label{eq: taylor varphi low temperature}
\varphi(t)=\varphi(t_{*})+\sum_{p=2}^{+\infty}a_{p}(t-t_{*})^p,
\end{equation}
where
\begin{equation}\label{eq: 2f_2 and func of t star}
2a_{2}=\frac{1}{\cosh^2 t_{*}}-\frac{1}{2d\beta}.
\end{equation}
By Lemma \ref{lem: the solutions of the mean-field equation} and Lemma \ref{lem: the maximum point of varphi} we know that $t_*>0$ is the solution of equation $\tanh t=t/(2d\beta)$. Together with $\sinh x>x$ for $x>0$, we have that
\begin{equation}\label{eq: supercritical 2f_2 < 0}
2a_{2}=\frac{1}{\cosh^2 t_{*}}-\frac{\tanh t_{*}}{t_{*}}=\frac{1}{2t_{*}\cosh^2 t_{*}}(2t_{*}-\sinh(2t_{*}))<0.
\end{equation}
Hence, by \eqref{eq: taylor varphi low temperature} and \eqref{eq: supercritical 2f_2 < 0} there exists $\delta\in(0,t_{*}/2)$ and $\kappa>0$ such that for $t\in [t_{*}-\delta,t_{*}+\delta]$, we have that $\varphi(t)<\varphi(t_{*})-\kappa(t-t_{*})^{2}$. Besides, there exists $\widetilde{\kappa}>0$ such that $\varphi(t)<\varphi(t_{*})-\widetilde{\kappa}(t-t_{*})$ for $t>t_{*}+\delta$ and $\varphi(t)<\varphi(t_{*})-\tilde{\kappa}$ for $t\in[0,t_{*}-\delta]$. Therefore, there exists $\widetilde{\delta}, c, C>0$ such that for all $0<x<\widetilde{\delta}<1$, we have that $|x\ln x|<\delta$, and
\begin{align}\label{eq: R> negligible}
R^{>}(x)e^{-\varphi(t_{*})/x^2}&\leq\frac{1}{\sqrt{\pi d\beta}x}\int_{0}^{t_{*}-\delta}e^{-\widetilde{\kappa}/x^2}\,\mathrm{d}t+\frac{1}{\sqrt{\pi d\beta}x}\int_{t_{*}-\delta}^{t_{*}+x\ln x}e^{-\kappa(t-t_{*})^2/x^2}\,\mathrm{d}t\notag\\
&\quad +\frac{1}{\sqrt{\pi d\beta}x}\int_{t_{*}-x\ln x}^{t_{*}+\delta}e^{-\kappa(t-t_{*})^2/x^2}\,\mathrm{d}t\notag\\
&\quad +\frac{1}{\sqrt{\pi d\beta}x}\int_{t_{*}+\delta}^{+\infty}e^{-\widetilde{\kappa}(t-t_{*})/x^2}\,\mathrm{d}t\notag\\
&\leq Ce^{-c(\ln|x|)^2}.
\end{align}
Hence, as $x\to 0$, $R^{>}(x)e^{-\varphi(t_{*})/x^2}$ is much smaller than  $x^{M}$  for any $M\geq 1$.

For the main part $Z^{>}(x)$, by \eqref{eq: taylor varphi low temperature} we have that
\begin{align}\label{eq: supercritical expansion of hat Z via integral}
Z^{>}(x)&=\frac{1}{\sqrt{\pi d\beta}x}\int_{t_{*}+x\ln x}^{t_{*}-x\ln x}e^{(\varphi(t_{*})+\sum_{p=2}^{+\infty}a_{p}(t-t_{*})^p)/x^2}\,\mathrm{d}t\notag\\
&\xlongequal{s=(t-t_*)/x} \frac{1}{\sqrt{\pi d\beta}}e^{\varphi(t_{*})/x^2}\int_{\ln x}^{-\ln x}e^{a_2 s^2}e^{\sum_{p\geq 3}a_ps^px^{p-2}}\,\mathrm{d}s\notag\\
&= \frac{1}{\sqrt{\pi d\beta}}e^{\varphi(t_{*})/x^2}\int_{\ln x}^{-\ln x}e^{a_2 s^2}e^{\sum_{p\geq 1}a_{p+2}s^{p+2}x^{p}}\,\mathrm{d}s\notag\\
&= \frac{1}{\sqrt{\pi d\beta}}e^{\varphi(t_{*})/x^2}\int_{\ln x}^{-\ln x}e^{a_2 s^2}(1+\sum_{p=1}^{M}b_p^{>}(s)x^{p}+r_{M}^{>}(s,x))\,\mathrm{d}s,
\end{align}
where
\begin{equation}\label{eq: gps}
b_p^{>}(s)=\sum_{k=1}^{p}\sum_{\substack{m_1,m_2,\ldots,m_{k}\geq 1\\ m_1+m_2+\cdots+m_{k}=p}}a_{m_1+2}a_{m_2+2}\cdots a_{m_k+2}s^{p+2k}/k!
\end{equation}
and there exists $0<C<+\infty$ such that for $|s|<-\ln x$, we have that
\begin{equation}\label{eq: upper bounds for hat RMsx}
|r_{M}^{>}(s,x)|\leq Cx^{M+1}(-\ln x)^{C}=o(x^{M}).
\end{equation}
Note that
\begin{align}\label{eq: truncated integral ef2s2sp}
\int_{\ln x}^{-\ln x}e^{a_2s^2}s^{p}\,\mathrm{d}s&=\int_{-\infty}^{+\infty}e^{a_2s^2}s^{p}\,\mathrm{d}s+o(x^{M})\\
&=1_{p\text{ is even}}\times \Gamma\left(\frac{p+1}{2}\right)(-a_2)^{-\frac{p+1}{2}}+o(x^{M}).
\end{align}
Hence, combining \eqref{eq: truncation Z(x) low temperature 01}, \eqref{eq: truncation Z(x) low temperature 02}, \eqref{eq: R> negligible}, \eqref{eq: supercritical expansion of hat Z via integral}, \eqref{eq: gps}, \eqref{eq: upper bounds for hat RMsx} and \eqref{eq: truncated integral ef2s2sp}, for each $M\geq 1$, we have that
\begin{equation*}
Z(x)=e^{\varphi(t_*)/x^2}(\sum_{p=0}^{M}e_p^{>}x^{2p}+o(x^{2M})),
\end{equation*}
where $e_0^{>}=(-d\beta a_2)^{-1/2}$ and for $p\geq 1$,
\begin{align}\label{eq: hp}
e_p^{>}&=(-d\beta a_2)^{-1/2}\sum_{k=1}^{2p}\frac{(2p+2k-1)!!}{k!}(-2a_2)^{-(p+k)}\notag\\
&\quad \times\sum_{\substack{m_1,m_2,\ldots,m_{k}\geq 1\\ m_1+m_2+\cdots+m_{k}=2p}}a_{m_1+2}a_{m_2+2}\cdots a_{m_k+2}.
\end{align}
Therefore, by \eqref{eq: Zdbetah via Z(x)} and \eqref{eq: 2f_2 and func of t star}, we have
\[Z_{d,\beta,0}=2^ne^{n\varphi(t_*)}(\sum_{p=0}^{M}e_p^{>}n^{-p}+o(n^{-M})),\]
where $\varphi(t_*)=\ln\cosh t_{*}-t_{*}^2/(4d\beta)$, $e_0^{>}=(-d\beta a_2)^{-\frac{1}{2}}=2^{\frac{1}{2}}(d\beta)^{-\frac{1}{2}}\left(\frac{1}{2d\beta}-\frac{1}{\cosh^2 t_{*}}\right)^{-\frac{1}{2}}$ and for $p\geq 1$, $e_p^{>}$ is given by \eqref{eq: hp}. Thus, Theorem \ref{thm: normalizing constant expansion}(3) holds.

\begin{remark}\label{rem: supercritical connection with the free energy2}
In \cite[Chapter~2]{FriedliVelenikMR3752129}, the function
\[f_{CW}(m)=-d\beta m^2+\frac{1-m}{2}\ln\frac{1-m}{2}+\frac{1+m}{2}\ln\frac{1+m}{2}, \,\,m\in[-1,1]\]
is called the free energy functional, which is very important for the Curie-Weiss models. By Lemma \ref{lem: the solutions of the mean-field equation} and Lemma \ref{lem: the maximum point of varphi} we get  $z_*=\tanh\left(2d\beta z_*\right)$ and $t_{*}=2d\beta z_{*}$. Hence, by some simple calculation we can get the connection between the free energy functional $f_{CW}(m)$ and the function $\varphi(t)$ as follows:
\[-f_{CW}(z_{*})=\ln 2+\varphi(t_*).\]
Note that $1/\cosh^2(t_*)=1-\tanh^2(t_*)=1-\tanh^2(2d\beta z_{*})=1-z^2_*$. So
 \[e_0^{>}=2^{1/2}(d\beta)^{-1/2}\left(\frac{1}{2d\beta}-\frac{1}{\cosh^2 t_{*}}\right)^{-\frac{1}{2}}=\left(\frac{1}{4}-\frac{d\beta}{2}(1-z^2_*)\right)^{-1/2}.\]
\end{remark}

\subsection{With non-zero external field}

Because $Z_{d,\beta,h}=Z_{d,\beta,|h|}$, we assume $h>0$ in this subsection. $Z(x)$ and $\varphi(t)$ are the same as in \eqref{eq: express Z(x) via an integral} and \eqref{eq: defn varphi t}. Again, we aim to find the expansion of $Z(x)$ at $x=0$. Hence, we restrict $x\in (0.001)$. By Lemma \ref{lem: the maximum point of varphi}, the maximal point of $\varphi(t)$ is $t_*$. So  by \eqref{eq: express Z(x) via an integral} and \eqref{eq: defn varphi t} we truncate $Z(x)$ as follows:
\begin{equation*}\label{eq: truncate Z with external field 01}
    Z(x)=Z^{h\neq 0}(x)+R^{h\neq 0}(x),
\end{equation*}
where
\begin{equation*}\label{eq: truncate Z with external field 02}
    Z^{h\neq 0}(x)=\frac{1}{2\sqrt{\pi d\beta}x}\int_{t_*+x\ln x}^{t_*-x\ln x}e^{\varphi(t)/x^2}\,\mathrm{d}t,\text{ and }R^{h\neq 0}(x)=Z(x)-Z^{h\neq 0}(x).
\end{equation*}

For $t$ sufficiently close to $t_*$, the Taylor expansion of $\varphi(t)$ is
\begin{equation}\label{eq: taylor varphi with external field}
\varphi(t)=\varphi(t_{*})+\sum_{p=2}^{+\infty}a_{p}(t-t_{*})^p,
\end{equation}
where
\begin{equation*}\label{eq: h > 0 2f_2 and func of t star}
2a_2=\frac{1}{\cosh^2 t_{*}}-\frac{1}{2d\beta}.
\end{equation*}
By Lemma \ref{lem: the solutions of the mean-field equation} and Lemma \ref{lem: the maximum point of varphi} we know that $t_*>0$ is the solution of equation $\tanh t=(t-h)/(2d\beta)$. Together with $\sinh x>x$ for $x>0$, we have that
\begin{equation*}
2a_{2}=\frac{1}{\cosh^2 t_{*}}-\frac{\tanh t_*}{t_*-h}<\frac{1}{\cosh^2 t_{*}}-\frac{\tanh t_*}{t_*}=\frac{1}{2t_{*}\cosh^2 t_{*}}(2t_{*}-\sinh(2t_{*}))<0.
\end{equation*}
Hence,  following the similar calculation in Subsection \ref{subsection: Z low temperature case} and by \eqref{eq: Zdbetah via Z(x)}, we obtain that
\begin{equation*}
    Z_{d,\beta,h}=2^ne^{n\varphi(t_*)}(\sum_{p=0}^{M}e_p^{h\neq 0}n^{-p}+o(n^{-M})).
\end{equation*}
Here, $\varphi(t_{*})=\ln\cosh t_{*}-\frac{(t_{*}-|h|)^2}{4d\beta}$,
$e_0^{h\neq 0}=(-4d\beta a_2)^{-\frac{1}{2}}=(2d\beta)^{-\frac{1}{2}}\left(\frac{1}{2d\beta}-\frac{1}{\cosh^2 t_{*}}\right)^{-\frac{1}{2}}$
and for $p\geq 1$,
\begin{align}\label{eq: h neq ep}
e_p^{h\neq 0} &= (-4d\beta a_2)^{-1/2}\sum_{k=1}^{2p}\frac{(2p+2k-1)!!}{k!}(-2a_2)^{-(p+k)}\notag\\
&\quad\times\sum_{\substack{m_1,m_2,\ldots,m_{k}\geq 1\\ m_1+m_2+\cdots+m_{k}=2p}}a_{m_1+2}a_{m_2+2}\cdots a_{m_k+2},
\end{align}
where $(a_p)_{p\geq 2}$ is given by \eqref{eq: taylor varphi with external field}. Thus Theorem \ref{thm: normalizing constant expansion}(4) holds.

\begin{remark}\label{rem: h neq 0 connection with the free energy}
Fix $h>0$. Recall the free energy functional
\begin{equation}\label{eq: free energy functional}
  f_{CW}(m)=-d\beta m^2-hm+\frac{1+m}{2}\ln\frac{1+m}{2}+\frac{1-m}{2}\ln\frac{1-m}{2}, \,\,m\in[-1,1]
\end{equation}
in \cite[Chapter~2]{FriedliVelenikMR3752129}. By Lemma \ref{lem: the solutions of the mean-field equation} and Lemma \ref{lem: the maximum point of varphi} we get  $z_*=\tanh\left(2d\beta z_*+h\right)$ and $t_{*}=2d\beta z_{*}+h$. Hence, by the similar calculation as in  Remark~\ref{rem: supercritical connection with the free energy2} we have that
\[-f_{CW}(z_*)=\varphi(t_*)+\ln 2, \,\,e_0^{h\neq 0}=(1-2d\beta(1-z_{*}^2))^{-\frac{1}{2}}.\]
\end{remark}

\begin{remark}
For $h=0$ and $\beta>\beta_c$ i.e. the low temperature case in Subsection \ref{subsection: Z low temperature case}, there are two maximum points $t_{*}$ and $-t_{*}$ for the function $\varphi(t)$. The contribution to $Z_{d,\beta,h}$ of the integral near these two maximum points are the same. While for $h\neq 0$, i.e. the non-zero external field case, the maximum point $t_{*}$ of $\varphi(t)$ is unique. Therefore, compared with $e_{p}^{h\neq 0}$ in \eqref{eq: h neq ep}, there is a factor $2$ in $e_p^{>}$ in \eqref{eq: hp}.
\end{remark}

\section{Positive Littlewood-Offord problems for Curie-Weiss models}\label{sect: positive Littlewood-Offord exact}

We consider the positive Littlewood-Offord problems for Curie-Weiss models in this section, i.e. we prove Theorem \ref{thm: positive Littlewood-Offord problems} and Theorem~\ref{thm: positive Littlewood-Offord problem asymptotics} in this section.

\subsection{Proof of Theorem \ref{thm: positive Littlewood-Offord problems}}

Note that the possible values of $\sum_{i=1}^{n}\varepsilon_i$ lie in a arithmetic progression with common difference $2$. Hence, for $x\in\mathbb{R}$, there is a unique $y$ in that arithmetic progression such that $\sum_{i=1}^{n}\varepsilon_i\in(x-1,x+1)$ if and only if $\sum_{i=1}^{n}\varepsilon_i=y$. Hence, the second equality in \eqref{eq: Q_n^+ equivalent expression} holds, i.e.,
\begin{equation*}
    \sup_{x\in\mathbb{R}}P(\sum^n_{i=1}\varepsilon_i\in(x-1,x+1))= \sup_{x\in\mathbb{R}}P(\sum^n_{i=1}\varepsilon_i=x).
\end{equation*}

Fix $n\geq 1$, $v_1,v_2,\ldots,v_n\geq 1$ and $x\in\mathbb{R}$. Define
\[\Sigma_x=\{\sigma\in\{-1,1\}^n:\sigma_1v_1+\sigma_2v_2+\cdots+\sigma_nv_n\in(x-1,x+1)\}.\]
For a configuration $\sigma$ of spins, let
\begin{equation}\label{eq: spin positive}
   \mathcal{P}(\sigma)=\{i=1,2,\ldots,n:\sigma_i=1\},\,\, \mathcal{P}(\varepsilon)=\{i=1,2,\dots,n: \varepsilon_i=1\}.
\end{equation}
Then $\mathcal{P}(\sigma)$ is the set of vertices with the positive spin, and we can see that $\mathcal{P}(\sigma)$ and $\sigma$ determine each other. And $\mathcal{P}(\varepsilon)=\{i=1,2,\dots,n: \varepsilon_i=1\}$ is random.

Denote by $\mathcal{A}_x=\{\mathcal{P}(\sigma):\sigma\in\Sigma_x\}$. Since $v_1,v_2,\ldots,v_n\geq 1$, for distinct $\mathcal{P}(\sigma)$ and $\mathcal{P}(\sigma')$ in $\mathcal{A}_x$, neither $\mathcal{P}(\sigma)\subset \mathcal{P}(\sigma')$, nor $\mathcal{P}(\sigma')\subset\mathcal{P}(\sigma)$. Indeed, $\mathcal{P}(\sigma)\subset \mathcal{P}(\sigma')$ implies that
\[\sum^n_{k=1}\sigma'_kv_k-\sum^n_{k=1}\sigma_kv_k=\sum_{k\in \mathcal{P}(\sigma')\setminus \mathcal{P}(\sigma)}2v_k\geq 2.\]
Hence, $\mathcal{P}(\sigma)$ and $\mathcal{P}(\sigma')$ cannot both belong to $\mathcal{A}_{x}$. Similarly, we cannot have $\mathcal{P}(\sigma')\subset\mathcal{P}(\sigma)$.

Such a system $\mathcal{A}_x$ without inclusion for different sets is called a \emph{Sperner system} or an \emph{anti-chain}. We refer to \cite[Definition~7.1]{TaoVuMR2289012} for the definition of anti-chains. By LYM inequality (see \cite[Lemma~7.2]{TaoVuMR2289012}), we have that
\begin{equation*}
1\geq \sum_{\mathcal{P}(\sigma)\in\mathcal{A}_x}1/\binom{n}{|\mathcal{P}(\sigma)|},
\end{equation*}
where $|\mathcal{P}(\sigma)|$ is the cardinality of $\mathcal{P}(\sigma)$. Then, we have that
\begin{align*}
1&\geq \sum_{\mathcal{P}(\sigma)\in\mathcal{A}_x}P(\varepsilon=\sigma)\frac{1}{\binom{n}{|\mathcal{P}(\sigma)|}P(\varepsilon=\sigma)}\\
&\geq \frac{1}{\max_{\sigma\in\Sigma_x}\binom{n}{|\mathcal{P}(\sigma)|}P(\varepsilon=\sigma)}\sum_{\mathcal{P}(\sigma)\in\mathcal{A}_x}P(\varepsilon=\sigma).
\end{align*}
Hence, we have that
\begin{multline*}
P(\varepsilon_1v_1+\varepsilon_2v_2+\dots+\varepsilon_nv_n\in(x-1,x+1))=\sum_{\mathcal{P}(\sigma)\in\mathcal{A}_x}P(\varepsilon=\sigma)\\
\leq\max_{\sigma\in\Sigma_x}\binom{n}{|\mathcal{P}(\sigma)|}P(\varepsilon=\sigma)\leq\max_{\sigma\in\{-1,1\}^n}\binom{n}{|\mathcal{P}(\sigma)|}P(\varepsilon=\sigma).
\end{multline*}
Since $\varepsilon$ is exchangeable, we have that
\[P(|\mathcal{P}(\varepsilon)|=|\mathcal{P}(\sigma)|)=\binom{n}{|\mathcal{P}(\sigma)|}P(\varepsilon=\sigma).\]
Hence, we have that
\begin{equation}\label{eq: positive LO 01}
P(\varepsilon_1v_1+\varepsilon_2v_2+\cdots+\varepsilon_nv_n\in(x-1,x+1))\leq\max_{k=0,1,2,\dots,n}P(|\mathcal{P}(\varepsilon)|=k).
\end{equation}
Note that
\[P(|\mathcal{P}(\varepsilon)|=k)=P(\sum^n_{i=1}\varepsilon_i=2k-n)\leq\sup_{x\in\mathbb{R}}P(\sum^n_{i=1}\varepsilon_i=x),\]
Hence, by \eqref{eq: positive LO 01} we have that
\begin{equation*}\label{eq: Q_n^+ leq max k}
Q_n^{+}\leq\sup_{x\in\mathbb{R}}P(\sum^n_{i=1}\varepsilon_i=x)=\sup_{x\in\mathbb{R}}P(\sum^n_{i=1}\varepsilon_i\in(x-1,x+1)).
\end{equation*}
By definition of $Q_n^{+}$, we have the opposite inequality
\[\sup_{x\in\mathbb{R}}P(\sum_{i=1}^{n}\varepsilon_i\in(x-1,x+1))\leq Q_n^{+}.\]
Thus, Theorem~\ref{thm: positive Littlewood-Offord problems} holds.

\begin{remark}\label{rem: pLO exchangeable}
By similar arguments, Theorem~\ref{thm: positive Littlewood-Offord problems} could be generalized to exchangeable random vectors with values in $\{-1,1\}^n$. I.e., let $\varepsilon=(\varepsilon_1,\varepsilon_2,\ldots,\varepsilon_n)$ be exchangeable random vectors with values in $\{-1,1\}^n$. Then, we have that
\begin{equation*}
Q^{+}_n=\sup_{x\in\mathbb{R}}P(\varepsilon_1+\varepsilon_2+\cdots+\varepsilon_n\in(x-1,x+1))
=\sup_{x\in\mathbb{R}}P(\varepsilon_1+\varepsilon_2+\cdots+\varepsilon_n=x).
\end{equation*}
\end{remark}

\subsection{Proof of Theorem~\ref{thm: positive Littlewood-Offord problem asymptotics}}\label{sect: positive Littlewood-Offord asymptotics}

By Theorem~\ref{thm: positive Littlewood-Offord problems} and the asymptotic expansion of $Z_{d,\beta,h}$ in Section~\ref{sect: asymptotic expansion of Z as a series in n}, it remains to obtain the asymptotic properties of
\[\max\limits_{k=0,1,\ldots,n}\binom{n}{k}\exp\left\{\frac{d\beta}{n}(2k-n)^2+h(2k-n)\right\}.\]

Because $\max\limits_{k=0,1,\dots,n}\binom{n}{k}\exp\left\{\frac{d\beta}{n}(2k-n)^2+h(2k-n)\right\}$ is an even function in $h$, we may assume $h\geq 0$ in the following.

Define the magnetization density
\[m:=-1+2k/n, \,\, k=0,1,\dots,n\]
and let $\mathcal{M}_n:=\{-1+\frac{2k}{n}:k=0,1,2,\dots,n\}\subset[-1,1]$.
Then, we have that
\begin{equation}\label{eq: positive k and m}
 \binom{n}{k}\exp\left\{\frac{d\beta}{n}(2k-n)^2+h(2k-n)\right\}=\binom{n}{\frac{1+m}{2}n}e^{d\beta nm^2+hnm}.
\end{equation}
For $-1\leq m\leq 1$, define the entropy density as the same as in \cite[Chapter~2]{FriedliVelenikMR3752129}:
\[s(m):=\left\{\begin{array}{ll}
-\frac{1+m}{2}\ln\frac{1+m}{2}-\frac{1-m}{2}\ln\frac{1-m}{2}, & \quad m\in(-1,1);\\
0, & \quad m=\pm 1.
\end{array}\right.\]
By the Stirling's formula, we have the following results:
\begin{lemma}\label{lem: binom bound}
There exist $0<c<C<+\infty$ such that for all $n\geq 2$ and $m\in\mathcal{M}_n$, it holds that
\begin{equation*}\label{eq: binom bound}
 c\exp\{ns(m)\}/\sqrt{n}\leq \binom{n}{\frac{1+m}{2}n}\leq C\exp\{ns(m)\}.
\end{equation*}
\end{lemma}
\begin{lemma}
Fix $\delta\in(0,1)$. We have that
\begin{equation*}
\lim_{n\to\infty}\max_{m\in\mathcal{M}_n\cap[1-\delta,1+\delta]}\left|\binom{n}{\frac{1+m}{2}n}/\left(\sqrt{\frac{2}{\pi n(1-m^2)}}\exp\{ns(m)\}\right)-1\right|=0.
\end{equation*}
\end{lemma}
According to the above lemmas, when $m$ is away from $-1$ and $1$, we get
\begin{equation}\label{eq: positive k and m 02}
  \binom{n}{\frac{1+m}{2}n}e^{d\beta nm^2+hnm}\sim\sqrt{\frac{2}{\pi n(1-m^2)}}\exp\{-nf_{CW}(m)\}, \,\,n\to+\infty
\end{equation}
where by \eqref{eq: free energy functional}, the free energy functional
\begin{equation*}
 f_{CW}(m)=-s(m)-d\beta m^2-hm, \,\, m\in[-1,1].
\end{equation*}
 Hence, the maximum of \eqref{eq: positive k and m} is closely related to the maximum point of $-f_{CW}(m)$. Note that
\[f'_{CW}(m)\leq 0 \Longleftrightarrow m\leq \tanh\left(\frac{\beta}{\beta_c}m+h\right).\]
So by Lemma \ref{lem: the solutions of the mean-field equation} we know that
for $h=0$ and $\beta\leq \beta_c$,  $z_*=0$ is the only maximum point of $-f_{CW}$. For $h=0$ and $\beta>\beta_c$,  $\pm z^*$ are the maximum points of $-f_{CW}$.For $h>0$, $z_*$ is the only maximum point of $-f_{CW}$.

Hence,
combining \eqref{eq: positive k and m}, \eqref{eq: positive k and m 02} and Lemma~\ref{lem: binom bound},  we have that as $n\to\infty$,
\[\max_{0\leq k\leq n} \binom{n}{k}\exp\left\{\frac{d\beta}{n}(2k-n)^2+h(2k-n)\right\}\sim\sqrt{\frac{2}{\pi n(1-z_{*}^2)}}e^{-f_{CW}(z_{*})n}.\]
Thus, by Theorem~\ref{thm: positive Littlewood-Offord problems}, we get that
\[Q^+_n\sim \frac{1}{Z_{d,\beta,h}}\sqrt{\frac{2}{\pi n(1-z_{*}^2)}}e^{-f_{CW}(z_{*})n},\text{ as }n\to \infty.\]
For $h=0$ and $\beta<\beta_c$, $f_{CW}(z_{*})=f_{CW}(0)=-\ln 2$. Together with Theorem~\ref{thm: normalizing constant expansion} (1), we get
\[Q^+_n\sim \sqrt{\frac{2(1-2d\beta)}{\pi}}n^{-\frac{1}{2}},\text{ as }n\to \infty.\]
For $h=0$ and $\beta=\beta_c$, $f_{CW}(z_*)=f_{CW}(0)=-\ln 2$. Together with Theorem~\ref{thm: normalizing constant expansion} (2) , we get
\[Q_n^{+}\sim \frac{2}{(\frac{3}{4})^{\frac{1}{4}}\Gamma(\frac{1}{4})}n^{-\frac{3}{4}},\text{ as }n\to\infty.\]
For $h=0$ and $\beta>\beta_c$,  by Theorem~\ref{thm: normalizing constant expansion} (3) and Remark~\ref{rem: supercritical connection with the free energy2}, we get
\[Q_n^{+}\sim\sqrt{\frac{1/(1-z_{*}^2)-2d\beta}{2\pi}}n^{-\frac{1}{2}},\text{ as }n\to\infty.\]
For $h\neq 0$,  by Theorem~\ref{thm: normalizing constant expansion} (4) and Remark~\ref{rem: h neq 0 connection with the free energy}, we get
\[Q_n^{+}\sim\sqrt{\frac{2(1/(1-z_{*}^2)-2d\beta)}{\pi}}n^{-\frac{1}{2}},\text{ as }n\to\infty.\qedhere\]
Thus, Theorem~\ref{thm: positive Littlewood-Offord problem asymptotics} holds.


\section{Littlewood-Offord problem for Bernoulli random variables}\label{sect: Littlewood-Offord for iid Bernoulli}
In this section, we consider the Littlewood-Offord problem for independent and identically distributed Bernoulli random variables.  Let $(\varepsilon_i)_{1\leq i\leq n}$ be independent and identically distributed under the probability measure $P_p$ with the distribution
\begin{equation}\label{eq: Bernoulli Pp}
  P_p(\varepsilon_1=1)=p, \quad P_p(\varepsilon_1=-1)=1-p, \quad 0\leq p\leq 1.
\end{equation}
We want to find an expression for the following concentration probability:
\begin{equation}\label{eq: Qnp}
Q_{n,p}:=\sup_{x\in\mathbb{R}}\sup_{|v_1|,|v_2|,\ldots,|v_n|\geq 1}P_p(\sum^n_{i=1}v_i\varepsilon_i\in(x-1,x+1)).
\end{equation}

In  \cite{JKMR4201801} and \cite{SinghalMR4440097}, people also study Littlewood-Offord problem for independent and identically distributed Bernoulli random variables. However, their setup is different as they consider
 \[\sup_{x\in\mathbb{R}}\sup_{v_i\neq 0,\forall i=1,2,\ldots,n}P(\sum_{i=1}^{n}v_i\varepsilon_i=x).\]
 It is obvious that
 \[\sup_{x\in\mathbb{R}}\sup_{v_i\neq 0,\forall i=1,2,\ldots,n}P(\sum_{i=1}^{n}v_i\varepsilon_i=x)\leq Q_{n,p}.\]
But it is not trivial to say that the equality holds.  We will show that the equality holds in the following Proposition \ref{prop: Qp}. And we study  $Q_{n,p}$  not only because it is an interesting problem itself, but also because it is important for the Littlewood-Offord problem for the Curie-Weiss models in Section \ref{sect: Littlewood-Offord exact and asymptotic}.

\begin{proposition}\label{prop: Qp}
\begin{equation*}
Q_{n,p}=\sup_{x\in\mathbb{R}}\sup_{|v_1|,|v_2|,\ldots,|v_n|\geq 1}P_p(\sum^n_{i=1}v_i\varepsilon_i=x)
=\sup_{\ell=0,1,2,\dots,n}\sup_{d\in\mathbb{Z}}P_p(\sum^{\ell}_{i=1}\varepsilon_i-\sum^n_{j=\ell+1}\varepsilon_j=d).
\end{equation*}
\end{proposition}
We will prove Proposition~\ref{prop: Qp} in Subsection~\ref{subsection: proof Qnp}.

\begin{corollary}\label{cor: even n bernoulli}
    For even $n$, we have that
    \[Q_{n,p}= P_p\left(\sum^{n/2}_{i=1}\varepsilon_i-\sum^n_{j=n/2+1}\varepsilon_j=0\right).\]
\end{corollary}
\begin{proof}
Singhal showed in \cite[Theorem~3.1]{SinghalMR4440097} that for even $n$,
\begin{equation*}
    \sup_{\ell=0,1,2,\dots,n}\sup_{d\in\mathbb{Z}}P_p(\sum^{\ell}_{i=1}\varepsilon_i-\sum^n_{j=\ell+1}\varepsilon_j=d)
    =P_p(\sum^{n/2}_{i=1}\varepsilon_i-\sum^n_{j=n/2+1}\varepsilon_j=0).
\end{equation*}
Then by Proposition~\ref{prop: Qp} we get the result.
\end{proof}

\begin{remark}
When $n$ is odd, the expression $\sup_{\ell=0,1,2,\dots,n}\sup_{d\in\mathbb{Z}}P_p(\sum^{\ell}_{i=1}\varepsilon_i-\sum^n_{j=\ell+1}\varepsilon_j=d)$ is complicated, see \cite[Subsection 3.3]{SinghalMR4440097}. So by  Proposition~\ref{prop: Qp}, when $n$ is odd, the expression of $Q_{n,p}$ is complicated and in our opinion it is not useful for our study in Section \ref{sect: Littlewood-Offord exact and asymptotic}. Hence, we do not write it down in this
 paper.
\end{remark}

\subsection{Proof of Proposition~\ref{prop: Qp}}\label{subsection: proof Qnp}

Some proof ideas come from \cite{SinghalMR4440097}.

Fix $v_1,v_2,\ldots,v_n$ such that $|v_1|,|v_2|,\ldots,|v_n|\geq 1$. Let
\[A=\{i: v_i\geq 1\}, \,\, B=\{i: v_i\leq -1\},\]
and
\begin{equation*}
S=\left\{\sum_{i\in A}v_i\sigma_i: \sigma_i=\pm 1,\forall i=1,2,\ldots,n\right\}, T=\left\{\sum_{i\in B}|v_i|\sigma_i: \sigma_i=\pm 1,\forall i=1,2,\ldots,n\right\}.
\end{equation*}
Then by independence between $(\varepsilon_i)_{i\in A}$ and $(\varepsilon_i)_{i\in B}$, we get that
\begin{align}\label{eq: Pp 01}
P_p(\sum^n_{i=1}v_i&\varepsilon_i\in(x-1,x+1))\notag\\
&=P_p(\sum_{i\in A}v_i\varepsilon_i-\sum_{i\in B}|v_i|\varepsilon_i\in(x-1,x+1))\notag\\
&=\sum_{s\in S, t\in T}P_p(\sum_{i\in A}v_i\varepsilon_i=s, \sum_{i\in B}|v_i|\varepsilon_i=t)I_{(x-1,x+1)}(s-t)\notag\\
&=\sum_{s\in S, t\in T}P_p(\sum_{i\in A}v_i\varepsilon_i=s)P_p(\sum_{i\in B}|v_i|\varepsilon_i=t)I_{(x-1,x+1)}(s-t).
\end{align}
Here, $I_{E}(t)$ is the indicator function of the set $E$. For simplicity of notation, we write $I_{t_0}(t)$ instead of $I_{\{t_0\}}(t)$ for a single-element set $E=\{t_0\}$. Let $\mathcal{P}(\varepsilon)$ be as in \eqref{eq: spin positive}.
Note that
\[P_p(\sum_{i\in A}v_i\varepsilon_i=s)=\sum_{k=0}^{|A|}\sum_{A_{+}\subset A:|A_{+}|=k}P_p(A\cap\mathcal{P}(\varepsilon)=A_{+})I_{\sum_{i\in A_{+}}v_i-\sum_{i\in A\setminus A_{+}}v_i}(s).\]
Assume that $|A_{+}|=k$. Then, we have that
\[P_p(A\cap\mathcal{P}(\varepsilon)=A_{+})=p^{k}(1-p)^{|A|-k}.\]
Let $S_A$ be the set of all the permutations on the set $A$ and $S_B$ be the set of all the permutations on the set $B$. Then, we have that
\begin{equation*}
I_{\sum_{i\in A_{+}}v_i-\sum_{i\in A\setminus A_{+}}v_i}(s)=\frac{1}{k!(|A|-k)!}\sum_{\sigma\in S_{A}}\prod_{\ell=1}^{k}I_{A_{+}}(\sigma(\ell))I_{\sum_{\ell=1}^{k}v_{\sigma(\ell)}-\sum_{r=k+1}^{|A|}v_{\sigma(r)}}(s),
\end{equation*}
where $k=|A_{+}|$. Therefore, we have that
\begin{align*}
P_p(\sum_{i\in A}v_i\varepsilon_i=s)&=\sum_{k=0}^{|A|}\sum_{A_{+}\subset A:|A_{+}|=k}\sum_{\sigma\in S_{A}}\frac{1}{k!(|A|-k)!}p^{k}(1-p)^{|A|-k}\notag\\
&\quad \prod_{\ell=1}^{k}I_{A_{+}}(\sigma(\ell))I_{\sum_{\ell=1}^{k}v_{\sigma(\ell)}-\sum_{r=k+1}^{|A|}v_{\sigma(r)}}(s).
\end{align*}
By firstly summing over $A_{+}$ with $|A_{+}|=k$, we obtain that
\begin{equation}\label{eq: Pp A s03}
P_p(\sum_{i\in A}v_i\varepsilon_i=s)=\frac{1}{|A|!}\sum_{\sigma\in S_{A}}\alpha_{\sigma}(s),
\end{equation}
where
\begin{equation*}
\alpha_{\sigma}(s):= \sum_{k=0}^{|A|}\binom{|A|}{k}p^{k}(1-p)^{|A|-k}I_{\sum_{\ell=1}^{k}v_{\sigma(\ell)}-\sum_{r=k+1}^{|A|}v_{\sigma(r)}}(s).
\end{equation*}
Similarily, we get
\begin{equation}\label{eq: Pp B t01}
P_p(\sum_{i\in B}|v_i|\varepsilon_i=t)=\frac{1}{|B|!}\sum_{\tau\in S_B}\beta_{\tau}(t),
\end{equation}
where
\begin{equation*}
\beta_{\tau}(t):=\sum_{m=0}^{|B|}\binom{|B|}{m}p^{m}(1-p)^{|B|-m}I_{\sum_{\ell=1}^{m}|v_{\tau(\ell)}|-\sum_{r=m+1}^{|B|}|v_{\tau(r)}|}(t).
\end{equation*}
Hence, by \eqref{eq: Pp 01}, \eqref{eq: Pp A s03} and \eqref{eq: Pp B t01}, we get that
\begin{align}\label{eq: Pp 02}
P_p(\sum^n_{i=1}v_i&\varepsilon_i\in(x-1,x+1))\notag\\
=&\frac{1}{|A|!|B|!}\sum_{\sigma\in S_{A}}\sum_{\tau\in S_{B}}\sum_{s\in S, t\in T}\alpha_{\sigma}(s)\beta_{\tau}(t)I_{(x-1,x+1)}(s-t)\notag\\
\leq &\max_{\sigma\in S_A, \tau\in S_B}\sum_{s\in S, t\in T}\alpha_{\sigma}(s)\beta_{\tau}(t)I_{(x-1,x+1)}(s-t)\notag\\
=&\max_{\sigma\in S_A, \tau\in S_B}\left\{\sum^{|A|}_{k=0}\sum^{|B|}_{m=0}f(k)g(m)I_{(x-1,x+1)}(s_k(\sigma)-t_m(\tau))\right\},
\end{align}
where
\begin{equation}\label{eq: defn f and g}
f(k)=\binom{|A|}{k}p^k(1-p)^{|A|-k}, \,\,g(m)=\binom{|B|}{m}p^m(1-p)^{|B|-m}
\end{equation}
and
\begin{equation*}
s_k(\sigma)=\sum^k_{l=1}v_{\sigma(l)}-\sum^{|A|}_{r=k+1}v_{\sigma(r)}, \,\, t_m(\tau)=\sum^m_{\ell=1}|v_{\tau(\ell)}|-\sum^{|B|}_{r=m+1}|v_{\tau(r)}|.
\end{equation*}
Note that
\begin{equation}\label{eq: sk tm}
s_{k+1}(\sigma)-s_k(\sigma)=2v_{\sigma(k+1)}\geq 2, \,\,t_{m+1}(\tau)-t_m(\tau)=2|v_{\tau(m+1)}|\geq 2.
\end{equation}
By \eqref{eq: Pp 02} and \eqref{eq: sk tm}, we get
\begin{multline}\label{eq: Pp 03}
P_p(\sum^n_{i=1}v_i\varepsilon_i\in(x-1,x+1))\\
\leq\max_{\substack{s_k, t_m\in\mathbb{R}, \\ s_{k+1}-s_k\geq 2, t_{m+1}-t_m\geq 2, \\ 0\leq k\leq |A|, 0\leq m\leq |B|}}\left\{\sum^{|A|}_{k=0}\sum^{|B|}_{m=0}f(k)g(m)I_{(x-1,x+1)}(s_k-t_m)\right\}.
\end{multline}

In the following, we use the language of graph theories and bound \eqref{eq: Pp 03} from above by the weighted sum of edges in certain non-crossing bipartite graphs.  Let us consider a bipartite graph with the vertex set $V=V_1\cup V_2$, where
\begin{equation*}
	V_1=\{(0,1),(1,1),\ldots,(|B|,1)\}, V_2=\{(0,2),(1,2),\ldots,(|A|,2)\}.
\end{equation*}
And we put an (undirected) edge (an line segment) between two vertices  $(m,1)$  and $(k,2)$ iff $s_k-t_m\in(x-1,x+1)$. Then let $E$ be the set of all the edges. Let $G=(V,E)$ be the graph. By \eqref{eq: sk tm}, two different edges do not cross each other, i.e. if $s_{k_1}-t_{m_1}\in (x-1,x+1)$, $s_{k_2}-t_{m_2}\in (x-1,x+1)$, then either $k_1<k_2$, $m_1<m_2$ or $k_1>k_2$, $m_1>m_2$. Indeed, if $k_1<k_2$ and $m_2\leq m_1$, then
\begin{equation*}
s_{k_2}-t_{m_2}=(s_{k_2}-s_{k_1})+(t_{m_1}-t_{m_2})+s_{k_1}-t_{m_1}\geq 2+s_{k_1}-t_{m_1}>2+(x-1)=x+1,
\end{equation*}
which is a contradiction. It is similar to rule out the other cases. For the edge $e=\{(m,1), (k,2)\}\in E$, we  denote
\[\pi_1(e)=m,\,\, \pi_2(e)=k.\]
We define the weight $w(e)$ of the edge $e$ by
\[w(e)=f(\pi_2(e))g(\pi_1(e))=f(k)g(m).\]
Define the weight $w(G)$ of the graph $G$ as the sum of weights of all edges:
\[w(G):=\sum_{e\in E}w(e)=\sum^{|A|}_{k=0}\sum^{|B|}_{m=0}f(k)g(m)I_{(x-1,x+1)}(s_k-t_m).\]
Hence, by \eqref{eq: Pp 03}, we get that
\begin{align}\label{eq: Pp 04}
P_p(\sum^n_{i=1}v_i\varepsilon_i\in(x-1,x+1))&\leq \max_{\substack{s_k, t_m\in\mathbb{R}, \\ s_{k+1}-s_k\geq 2, t_{m+1}-t_m\geq 2, \\ 0\leq k\leq |A|, 0\leq m\leq |B|}}w(G)\notag\\
&\leq \sup_{\substack{G \,\,\text{is a non-crossing bipartite}\\
\text{graph with the vertex set $V$}}}w(G).
\end{align}

\begin{lemma}\label{lem: graph parallel}
There exists a non-crossing bipartite graph $G^*=(V, E^*)$ with the vertex set $V$ such that
$$w(G^*)=\sup_{\substack{G \,\,\text{is a non-crossing bipartite}\\
\text{graph with the vertex set $V$}}}w(G).$$
Besides, the edges (i.e. line segments) of $G^*$ are parallel to each other, i.e. there exists a constant $d$ such that for all $e\in E^*$,
\begin{equation*}
\pi_2(e)-\pi_1(e)=d.
\end{equation*}
\end{lemma}
\begin{proof}
Firstly, because the vertex set $V$ is finite, the non-crossing bipartite graphs with vertex set $V$ are finite. So, the set
\[\{w(G): G\text{ is a non-crossing bipartite graph with vertex set }V\}\]
is finite. Hence, there exists a non-crossing bipartite graph $G^*=(V, E^*)$ with the vertex set $V$ such that
$$w(G^*)=\sup_{\substack{G \,\,\text{is a non-crossing bipartite}\\
\text{graph with the vertex set $V$}}}w(G).$$

Secondly, we prove that the edges (line segments) of $G^*$ are parallel to each other in the following. Consider two neighbour edges $e_1,e_2\in E^*$ with $\pi_2(e_1)<\pi_2(e_2)$. Because $G^*$ is a non-crossing bipartite graph and $\pi_2(e_2)>\pi_2(e_1)$, it holds that $\pi_1(e_2)>\pi_1(e_1)$. Necessarily, we have that
\begin{equation}\label{eq: the space between e1 and e2}
  \pi_2(e_2)-\pi_2(e_1)=1\,\,\text{or}\,\,\pi_1(e_2)-\pi_1(e_1)=1.
\end{equation}
Because when $\pi_2(e_2)-\pi_2(e_1)\geq 2$ and $\pi_1(e_2)-\pi_1(e_1)\geq 2$, we can add an edge $\widetilde{e}$ in  the graph $G^*$ such that $\pi_2(e_1)<\pi_2(\widetilde{e})<\pi_2(e_2)$ and $\pi_1(e_1)<\pi_1(\widetilde{e})<\pi_1(e_2)$. Let $\widetilde{E}=\{\widetilde{e}\}\cup E^*$ and $\widetilde{G}=(V, \widetilde{E})$. Then, since the weight of each edge is positive, $\widetilde{G}$ is a non-crossing bipartite graph with
\begin{equation*}
w(\widetilde{G})-w(G^*)=\sum_{e\in \widetilde{E}}f(\pi_2(e))g(\pi_1(e))-\sum_{e\in E^*}f(\pi_2(e))g(\pi_1(e))=f(\pi_2(\widetilde{e}))g(\pi_1(\widetilde{e}))>0,
\end{equation*}
which contradicts with the maximality of $w(G^{*})$. Hence, \eqref{eq: the space between e1 and e2} holds.

When $\pi_2(e_2)-\pi_2(e_1)=1$, we want to prove $\pi_1(e_2)-\pi_1(e_1)=1$. If $\pi_1(e_2)-\pi_1(e_1)\neq1$,
we have $\pi_1(e_2)-\pi_1(e_1)\geq 2$. By the definition of the functions $f$ and $g$ in \eqref{eq: defn f and g}, we know that the functions $f$ and $g$ are unimodal. Hence, there are only two cases for the function $g$:

Case 1. There exists $\pi_1(e_1)< j_0<\pi_1(e_2)$ such that
$$g(j_0)>g(\pi_1(e_2)).$$
Then let $\widetilde{e}:=\{(j_0,1),(\pi_2(e_2),2)\}$, $\widetilde{E}=\{\widetilde{e}\}\cup \{e\in E^*: e\neq e_2\}$ and $\widetilde{G}=(V, \widetilde{E})$. Hence, we get that
\begin{align*}
w(\widetilde{G})-w(G^*)&=\sum_{e\in\widetilde{E}}f(\pi_2(e))g(\pi_1(e))-\sum_{e\in E^*}f(\pi_2(e))g(\pi_1(e))\\
&=f(\pi_2(\widetilde{e}))g(\pi_1(\widetilde{e}))-f(\pi_2(e_2))g(\pi_1(e_2))\\
&=f(\pi_2(e_2))[g(j_0)-g(\pi_1(e_2))]>0,
\end{align*}
which contradicts with the maximality of $w(G^{*})$.

Case 2. There exists no $\pi_1(e_1)< j_0<\pi_1(e_2)$ such that
$$g(j_0)>g(\pi_1(e_2)).$$
Then, since $g$ is unimodal, we must have that
\[g(\pi_1(e_1))<g(\pi_1(e_1)+1)<\dots<g(\pi_1(e_2)-1)\leq g(\pi_1(e_2)).\]
Then let $\widetilde{e}:=\{(\pi_1(e_1)+1,1),(\pi_2(e_1),2)\}$, $\widetilde{E}=\{\widetilde{e}\}\cup \{e\in E^*: e\neq e_1\}$ and $\widetilde{G}=(V, \widetilde{E})$. Hence, we get that
\begin{align*}
w(\widetilde{G})-w(G^*)&=\sum_{e\in\widetilde{E}}f(\pi_2(e))g(\pi_1(e))-\sum_{e\in E^*}f(\pi_2(e))g(\pi_1(e))\\
&=f(\pi_2(\widetilde{e}))g(\pi_1(\widetilde{e}))-f(\pi_2(e_1))g(\pi_1(e_1))\\
&=f(\pi_2(e_1))[g(\pi_1(e_1)+1)-g(\pi_1(e_1))]>0,
\end{align*}
which contradicts with the maximality of $w(G^{*})$.

Combining Case 1 and Case 2, we see that $\pi_1(e_2)-\pi_1(e_1)=1$ as long as $\pi_2(e_2)-\pi_2(e_1)=1$. Similarly, since the function $f(k)$ is also unimodal, we have that $\pi_2(e_2)-\pi_2(e_1)=1$ as long as $\pi_1(e_2)-\pi_1(e_1)=1$. Hence, for any pair of neighbour edges $e_1,e_2\in E^*$, we have that $\pi_2(e_2)-\pi_2(e_1)=1$ and $\pi_1(e_2)-\pi_1(e_1)=1$, which means that the edges (line segments) of $G^*$ are parallel to each other.
\end{proof}

By Lemma \ref{lem: graph parallel}, there exists a constant $d$ such that
\begin{align}\label{eq: wG* and Pp}
	w(G^*)&=\sum_{e\in E^*}f(\pi_2(e))g(\pi_1(e))\notag\\
    &=\sum_{\substack{0\leq k\leq |A|, 0\leq m\leq |B|\\ k-m=d}}f(k)g(m)\notag\\
	&=\sum_{\substack{0\leq k\leq |A|, 0\leq m\leq |B|\\ k-m=d}}P_p(\sum^{|A|}_{i=1}\varepsilon_i=k)P_p(\sum^n_{j=|A|+1}\varepsilon_j=m)\notag\\
	&=P_p(\sum^{|A|}_{i=1}\varepsilon_i-\sum^n_{j=|A|+1}\varepsilon_j=d).
\end{align}
So, by \eqref{eq: Pp 04}, \eqref{eq: wG* and Pp} and Lemma \ref{lem: graph parallel}, we get
\begin{align*}
	P_p(\sum^n_{i=1}v_i\varepsilon_i\in (x-1, x+1))\leq& P_p(\sum^{|A|}_{i=1}\varepsilon_i-\sum^n_{j=|A|+1}\varepsilon_j=d)\\
	\leq& \sup_{d\in\mathbb{Z}}P_p(\sum^{|A|}_{i=1}\varepsilon_i-\sum^n_{j=|A|+1}\varepsilon_j=d)\\
	\leq& \sup_{\ell=0,1,2,\dots,n}\sup_{d\in\mathbb{Z}}P_p(\sum^{\ell}_{i=1}\varepsilon_i-\sum^n_{j=\ell+1}\varepsilon_j=d).
\end{align*}
Hence, by \eqref{eq: Qnp}, we have that
\begin{align*}
Q_{n,p}&\leq\sup_{\ell=0,1,2,\dots,n}\sup_{d\in\mathbb{Z}}P_p(\sum^{\ell}_{i=1}\varepsilon_i-\sum^n_{j=\ell+1}\varepsilon_j=d)\\
&\leq \sup_{\ell=0,1,2,\dots,n}\sup_{d\in\mathbb{Z}}P_p(\sum^{\ell}_{i=1}\varepsilon_i-\sum^n_{j=\ell+1}\varepsilon_j\in (d-1,d+1))\\
&\leq Q_{n,p},
\end{align*}
i.e.
\begin{equation*}\label{eq: Qnp expression in d}
	Q_{n,p}= \sup_{\ell=0,1,2,\dots,n}\sup_{d\in\mathbb{Z}}P_p(\sum^{\ell}_{i=1}\varepsilon_i-\sum^n_{j=\ell+1}\varepsilon_j=d).
\end{equation*}
Hence, by \cite[Theorem~1.2]{SinghalMR4440097}, it holds that
\[Q_{n,p}= \sup_{\ell=0,1,2,\dots,n}\sup_{d\in\mathbb{Z}}P_p(\sum^{\ell}_{i=1}\varepsilon_i-\sum^n_{j=\ell+1}\varepsilon_j=d)=\sup_{x\in\mathbb{R}}\sup_{|v_1|,|v_2|,\ldots,|v_n|\geq 1}P_p(\sum^n_{i=1}v_i\varepsilon_i=x).\] Hence, Proposition~\ref{prop: Qp} holds.

\section{Littlewood-Offord problems for Curie-Weiss models}\label{sect: Littlewood-Offord exact and asymptotic}
In this section we study the Littlewood-Offord problems for Curie-Weiss models, i.e. we  prove Theorem~\ref{thm: Littlewood-Offord problems} and Theorem~\ref{thm: Littlewood-Offord problem asymptotics}.

\subsection{Proof of Theorem~\ref{thm: Littlewood-Offord problems}}
\begin{lemma}\label{lem: Curie weiss and Bernoulli}
Let $\varepsilon=(\varepsilon_1, \varepsilon_2,\ldots,\varepsilon_n)$ be the vector of spins in Curie-Weiss models. Then for any $\sigma\in\{-1,1\}^n$,
\[P(\varepsilon=\sigma)=\int_{0}^{1}K(p,\sigma)\nu(\mathrm{d}p),\]
where $\nu(\mathrm{d}p)$ is a probability measure, $K(p,\sigma)$ is a probability kernel such that for fixed $p\in (0,1)$,
\begin{equation}\label{eq: defn Kpsigma}
K(p,\sigma)=\prod_{i=1}^{n}(p\cdot I_{\{1\}}(\sigma_i)+(1-p)\cdot I_{\{-1\}}(\sigma_i)),
\end{equation}
i.e. under the probability measure $K(p,\sigma)$ for fixed $p\in (0,1)$,
$(\varepsilon_i)_{1\leq i\leq n}$ are independent and identically distributed Bernoulli random variables with parameter $p$.
\end{lemma}
\begin{proof}
By Hubbard's transformation, we have that
\begin{equation*}
 P(\varepsilon=\sigma)=\frac{1}{Z_{d,\beta,h}}\exp\left\{\frac{d\beta}{n}(\sum^n_{j=1}\sigma_j)^2+h\sum^n_{j=1}\sigma_j\right\}=\frac{1}{Z_{d,\beta,h}}E\left[e^{(\sum^n_{j=1}\sigma_j)(\sqrt{\frac{2d\beta}{n}}Y+h)}\right],
\end{equation*}
where $Y$ is a standard Gaussian random variable. Thus, we have that
\begin{align}\label{eq: varepsilon conditional disitribution}
 P(\varepsilon=\sigma)=&\frac{1}{Z_{d,\beta,h}}\int_{\mathbb{R}}e^{(\sum^n_{j=1}\sigma_j)(\sqrt{\frac{2d\beta}{n}}y+h)}\frac{1}{\sqrt{2\pi}}e^{-y^2/2}\,\mathrm{d}y\notag\\
 =&\frac{1}{Z_{d,\beta,h}}\int_{\mathbb{R}}2^n\cosh^n\left(\sqrt{\frac{2d\beta}{n}y}+h\right)\frac{1}{\sqrt{2\pi}}e^{-y^2/2}\notag\\
&\quad\quad\prod^n_{j=1}\frac{e^{\sigma_j(\sqrt{\frac{2d\beta}{n}}y+h)}}{e^{(\sqrt{\frac{2d\beta}{n}}y+h)}+e^{-(\sqrt{\frac{2d\beta}{n}}y+h)}}\,\mathrm{d}y.
\end{align}
Let $p=e^{(\sqrt{\frac{2d\beta}{n}}y+h)}/(e^{(\sqrt{\frac{2d\beta}{n}}y+h)}+e^{-(\sqrt{\frac{2d\beta}{n}}y+h)})$. Then, we have that
\[P(\varepsilon=\sigma)=\int_{0}^{1}K(p,\sigma)\nu(\mathrm{d}p),\]
where
\begin{equation*}
\nu(\mathrm{d}p)=\frac{\mathrm{d}p}{Z_{d,\beta,h}}\sqrt{\frac{n}{\pi d\beta}}\frac{1}{4p(1-p)} [p(1-p)]^{-\frac{n}{2}}\exp\left\{-\frac{n}{4d\beta}\left(\frac{1}{2}\ln\left(\frac{p}{1-p}\right)-h\right)^2\right\},
\end{equation*}
which is a probability measure, and $K(p,\sigma)$ is defined in \eqref{eq: defn Kpsigma}.
\end{proof}

Hence, from Lemma~\ref{lem: Curie weiss and Bernoulli}, we get that the Curie-Weiss model is a mixture of independent and identically distributed Bernoulli random variables.
Recall the definitions of $Q_n$, $P_p(\,\cdot\,)$ and $Q_{n,p}$ respectively in \eqref{eq: def Qn}, \eqref{eq: Bernoulli Pp}  and \eqref{eq: Qnp}. By Lemma~\ref{lem: Curie weiss and Bernoulli}, we have that
\begin{align}\label{eq: Qn inequ01}
Q_n&=\sup_{x\in\mathbb{R}}\sup_{|v_1|,|v_2|,\ldots,|v_n|\geq 1}\int^1_0P_p(\sum_{i=1}^{n}\varepsilon_iv_i\in(x-1,x+1))\nu(\mathrm{d}p)\notag\\
&\leq \int^1_0\sup_{x\in\mathbb{R}}\sup_{|v_1|,|v_2|,\ldots,|v_n|\geq 1}P_p(\sum_{i=1}^{n}\varepsilon_iv_i\in(x-1,x+1))\nu(\mathrm{d}p)\notag\\
&=\int^1_0Q_{n,p}\nu(\mathrm{d}p).
\end{align}
Hence, if $n$ is even, by \eqref{eq: Qn inequ01} and Corollary~\ref{cor: even n bernoulli}, we get
\begin{multline*}
P(\sum^{n/2}_{i=1}\varepsilon_i-\sum^n_{j=n/2+1}\varepsilon_j=0)\leq Q_n\leq \int^1_0Q_{n,p}\nu(\mathrm{d}p)\\
=\int^1_0P_p(\sum^{n/2}_{i=1}\varepsilon_i-\sum^n_{j=n/2+1}\varepsilon_j=0)\nu(\mathrm{d}p)=P(\sum^{n/2}_{i=1}\varepsilon_i-\sum^n_{j=n/2+1}\varepsilon_j=0),
\end{multline*}
i.e.
\begin{equation}\label{eq: Qn even}
Q_n=\int_{0}^{1}Q_{n,p}\,\nu(\mathrm{d}p)=P(\sum^{n/2}_{i=1}\varepsilon_i-\sum^n_{j=n/2+1}\varepsilon_j=0).
\end{equation}
By rewriting $\sum_{i=1}^{n}\varepsilon_iv_i\in(x-1,x+1)$ as
$\sum_{i=1}^{n-1}\varepsilon_iv_i\in(x-1-\varepsilon_nv_n,x+1-\varepsilon_nv_n)$,
using the independence between $(\varepsilon_i)_{1\leq i\leq n}$ under $P_p$, we see that
\[Q_{n,p}\leq Q_{n-1,p}.\]
Hence, for odd $n$, using \eqref{eq: Qn inequ01}, together with \eqref{eq: Qn even} for even $n-1$, we have that
\begin{equation*}
P\left(\sum_{i=1}^{(n-1)/2}\varepsilon_i-\sum_{i=(n+1)/2}^{n}\varepsilon_i=1\right)\leq Q_n
\leq\int_{0}^{1}Q_{n,p}\,\nu(\mathrm{d}p)\leq\int_{0}^{1}Q_{n-1,p}\,\nu(\mathrm{d}p)=Q_{n-1}.
\end{equation*}
Thus Theorem~\ref{thm: Littlewood-Offord problems} holds.

\subsection{Proof of Theorem~\ref{thm: Littlewood-Offord problem asymptotics}}
From Theorem~\ref{thm: Littlewood-Offord problems}, we can see that for even $n$, $Q_n$ has a specific expression, but for odd $n$, we only get a inequality. Hence, we have to discuss the oddity of $n$ to prove Theorem~\ref{thm: Littlewood-Offord problem asymptotics} in the following.

\subsubsection{For even \texorpdfstring{$n$}{n}}\label{subsubsect: even n}

Let $n\geq 2$ be a positive even integer. Define
\begin{align*}
O_{d,\beta,h} &= \sum_{\sigma\in\{-1,1\}^n:\sum_{i=1}^{n/2}\sigma_i-\sum_{i=n/2+1}^{n}\sigma_i=0}\exp\left\{\frac{d\beta}{n}(\sum_{j=1}^{n}\sigma_j)^2+h\sum_{j=1}^{n}\sigma_j\right\}\\
\end{align*}
Then, by \eqref{eq: Qn even equality}, we have that
\begin{equation}\label{eq: Qn = O/Z}
Q_n=P\left(\sum_{i=1}^{n/2}\varepsilon_i-\sum_{i=n/2+1}^{n}\varepsilon_i=0\right)=\frac{O_{d,\beta,h}}{Z_{d,\beta,h}}.
\end{equation}

We wish to find the asymptotic expansion of $O_{d,\beta,h}$. Note that $O_{d,\beta,h}=O_{d,\beta,|h|}$. Hence, without loss of generality, we still assume that $h\geq 0$. Inspired by the method to get the asymptotic expansion of $Z_{d,\beta,h}$ in Section \ref{sect: asymptotic expansion of Z as a series in n}, we will expand $O_{d,\beta,h}$ in a similar way.

\begin{lemma}\label{lem: expectation representation of O d beta h}
Let $n\geq 2$ be a positive even integer. Then, we have that
\begin{equation*}
O_{d,\beta,h}=2^{n/2}E\left[\exp\left\{\frac{n}{2}\ln\left(\cosh(2\sqrt{\frac{2d\beta}{n}}Y+2h)+\cos U\right)\right\}\right],
\end{equation*}
where $Y$ is a standard Gaussian random variable, $U$ follows the uniform distribution on $(-\pi,\pi)$,  $Y$ and $U$ are independent.
\end{lemma}
\begin{proof}
Let $Y$ be a standard Gaussian random variable. By \eqref{eq: varepsilon conditional disitribution}, we can couple $\varepsilon$ with $Y$ in such a manner that
\begin{equation}\label{eq: varepsilon conditional disitribution with Y}
P(\varepsilon=\sigma)=\frac{1}{Z_{d,\beta,h}}E\left[2^n\cosh^n\left(\sqrt{\frac{2d\beta}{n}}Y+h\right)P(\varepsilon=\sigma|Y)\right],
\end{equation}
and conditionally on $Y=y$, $\varepsilon_1, \varepsilon_2,\ldots, \varepsilon_n$ are independent and identically distributed with the conditional distribution
\begin{equation}\label{eq: varepsilon conditional distribution}
P(\varepsilon_i=1|Y=y)=1-P(\varepsilon_i=-1|Y=y)=\frac{e^{(\sqrt{\frac{2d\beta}{n}}y+h)}}{e^{(\sqrt{\frac{2d\beta}{n}}y+h)}+e^{-(\sqrt{\frac{2d\beta}{n}}y+h)}}.
\end{equation}
Let $T_n=\sum^{n/2}_{i=1}\varepsilon_i-\sum^{n}_{i=n/2+1}\varepsilon_i$. By \eqref{eq: Qn = O/Z} and \eqref{eq: varepsilon conditional disitribution with Y}, we have that
\begin{align}\label{eq: O d beta h calculate01}
O_{d,\beta,h}&=Z_{d,\beta,h}P(T_n=0)\notag\\
&=\int_{\mathbb{R}}P(T_n=0|Y=y)
2^n\cosh^n\left(\sqrt{\frac{2d\beta}{n}}y+h\right)\frac{1}{\sqrt{2\pi}}e^{-y^2/2}\,\mathrm{d}y.
\end{align}
The conditional characteristic function of $T_n$ is
\begin{align*}
\varphi_{T_n|Y}(t|y)&:=E(e^{itT_n}|Y=y)\\
&=\left[\cos^2 t+\tanh^2\left(\sqrt{\frac{2d\beta}{n}}y+h\right)\sin^2t\right]^{n/2}\\
&=\frac{1}{2^{n/2}\cosh^n(\sqrt{\frac{2d\beta}{n}}y+h)}\left[\cosh\left(2\left(\sqrt{\frac{2d\beta}{n}}y+h\right)\right)+\cos (2t)\right]^{n/2}.
\end{align*}
As $T_n$ is integer-valued, we have that
 \begin{align*}\label{eq: P T_n = 0}
&\quad P(T_n=0 |Y=y)=\frac{1}{2\pi}\int_{-\pi}^{\pi}\varphi_{T_n|Y}(t|y)\,\mathrm{d}t\notag\\
&=\frac{1}{2^{n/2}\cosh^{n}(\sqrt{\frac{2d\beta}{n}}y+h)}\frac{1}{2\pi}\int_{-\pi}^{\pi}\left(\cosh\left(2\sqrt{\frac{2d\beta}{n}}y+2h\right)+\cos(2t)\right)^{n/2}\,\mathrm{d}t\notag\\
&\xlongequal{u=2t}\frac{1}{2^{n/2}\cosh^{n}(\sqrt{\frac{2d\beta}{n}}y+h)}\frac{1}{4\pi}\int_{-2\pi}^{2\pi}\left(\cosh\left(2\sqrt{\frac{2d\beta}{n}}y+2h\right)+\cos u\right)^{n/2}\,\mathrm{d}u\notag\\
&=\frac{1}{2^{n/2}\cosh^{n}(\sqrt{\frac{2d\beta}{n}}y+h)}\frac{1}{2\pi}\int_{-\pi}^{\pi}\left(\cosh\left(2\sqrt{\frac{2d\beta}{n}}y+2h\right)+\cos u\right)^{n/2}\,\mathrm{d}u.
\end{align*}
Hence, by \eqref{eq: O d beta h calculate01}, we get that
\begin{align*}
    O_{d,\beta,h}
&=2^{n/2}\int_{\mathbb{R}}\int^{\pi}_{-\pi}\left(\cosh\left(2\sqrt{\frac{2d\beta}{n}}y+2h\right)+\cos u\right)^{n/2}\frac{1}{(2\pi)^{3/2}}e^{-\frac{y^2}{2}}\,\mathrm{d}u\mathrm{d}y.\\
&=2^{n/2}E\left(\exp\left\{\frac{n}{2}\ln\left(\cosh(2\sqrt{\frac{2d\beta}{n}}Y+2h)+\cos U\right)\right\}\right),
\end{align*}
where $U$ is independent of $Y$ and $U$ follows the uniform distribution on $(-\pi,\pi)$.
\end{proof}

Secondly, we get the asymptotic expansion of $O_{d,\beta,h}$ in the following important proposition.
\begin{proposition}\label{prop: O asymptotics even}
Let $n$ be an even number. As $n\to+\infty$, the asymptotic expansion of $O_{d,\beta,h}$ is as the four cases below:
\begin{enumerate}[(1)]
  \item For $h=0$ and $\beta<\beta_c$, we have that
  \begin{equation*}
  O_{d,\beta,h}=2^n\left(\sum^{M}_{p=0}\gamma_p^{<}n^{-p-\frac{1}{2}}+o(n^{-M-\frac{1}{2}})\right),
  \end{equation*}
  where $\gamma_0^{<}=\sqrt{\frac{2}{\pi}}(1-2d\beta)^{-\frac{1}{2}}$, and for $p\geq 1$,  $\gamma_p^{<}$ is given by \eqref{eq: gamma p two cases}.
  \item For $h=0$ and $\beta=\beta_c$, we have that
  \begin{equation*}
  O_{d,\beta,h}=2^n\left(\sum^{M}_{p=0}\gamma_p^{c}n^{-\frac{p}{2}-\frac{1}{4}}+O(n^{-\frac{M}{2}-\frac{1}{4}})\right),
  \end{equation*}
  where $\gamma_0^{c}=\frac{1}{\pi}\left(\frac{3}{4}\right)^{\frac{1}{4}}\Gamma(\frac{1}{4})$ and  for $p\geq 1$, $\gamma^c_p$ is given by \eqref{eq: critical gamma p}. 
  \item For $h=0$ and $\beta>\beta_c$, we have that
  \begin{equation*}
  O_{d,\beta,h}=2^ne^{n\varphi(t_*)}\left(\sum^{M}_{p=0}\gamma_p^{>}n^{-p-\frac{1}{2}}+o(n^{-M-\frac{1}{2}})\right),
  \end{equation*}
  where $\varphi(t_*)=\ln \cosh{t_*}-\frac{t^2_*}{4d\beta}$  and $z_{*}=t_*/(2d\beta)$ is the maximal solution of the mean-field equation \[\tanh\left(\frac{\beta}{\beta_c} z\right)=z\]
  Here,
  $\gamma_0^{>}=2\sqrt{\frac{2}{\pi}}(2d\beta)^{-1/2}\left(\frac{1}{2d\beta}-\frac{1}{\cosh^2 t_{*}}\right)^{-1/2}\cosh{t_*}$ and for $p\geq 1$, $\gamma^>_p$ is given by \eqref{eq: gamma p low teperature}.
  \item For $h\neq 0$, we have that
  \begin{equation*}
  O_{d,\beta,h}=2^ne^{n\varphi(t_*)}\left(\sum_{p=0}^{M}\gamma_p^{h\neq 0}n^{-p-1/2}+o(n^{-M-\frac{1}{2}})\right),
  \end{equation*}
  where $\varphi(t_*)=\ln\cosh t_*-\frac{(t_*-|h|)^2}{4d\beta}$ and $z_{*}=(t_{*}-|h|)/(2d\beta)$ is the maximal solution of the mean-field equation
  \[\tanh\left(\frac{\beta}{\beta_c} z+|h|\right)=z.\]
  Here, $\gamma_0^{h\neq 0}=\sqrt{\frac{2}{\pi}}(2d\beta)^{-\frac{1}{2}}\left(\frac{1}{2d\beta}-\frac{1}{\cosh^2 t_{*}}\right)^{-1/2}\cosh{t_*}$, and for $p\geq 1$, $\gamma^{h\neq 0}_p$ is given by \eqref{eq: gamma p two cases}.
\end{enumerate}
\end{proposition}
\begin{proof}
    We define
\begin{equation*}
W(x)=E\left(\exp\left\{\frac{1}{2x^2}
\ln\left(\frac{1}{2}\cosh(2\sqrt{2d\beta}Yx+2h)+\frac{1}{2}\cos U\right)\right\}\right).
\end{equation*}
Then, by Lemma~\ref{lem: expectation representation of O d beta h} we have that
\begin{equation}\label{eq: O even via W}
O_{d,\beta,h}=2^{n}W(1/\sqrt{n}).
\end{equation}
In order to find the asymptotic expansion of $O_{d,\beta,h}$ as $n\to +\infty$, it suffices to find the asymptotic expansion of $W(x)$  as  $x\to 0^+$. For $x\neq 0$, we have that
\begin{align*}
W(x)
&=\frac{1}{(2\pi)^{\frac{3}{2}}}\int_{-\infty}^{+\infty}\int_{-\pi}^{\pi}e^{\frac{1}{2x^2}\ln(\frac{1}{2}\cosh(2\sqrt{2d\beta}yx+2h)+\frac{1}{2}\cos u)}e^{-y^2/2}\,\mathrm{d}y\mathrm{d}u\\
&\xlongequal{t=\sqrt{2d\beta}yx+h}\frac{1}{(2\pi)^{\frac{3}{2}}\sqrt{2d\beta}|x|}\int_{-\infty}^{+\infty}\int_{-\pi}^{\pi}e^{\psi(t,u)/x^2}\,\mathrm{d}t\mathrm{d}u,
\end{align*}
where
\begin{equation}\label{eq: defn psi}
\psi(t,u)=\frac{1}{2}\left(\ln\left(\frac{1}{2}\cosh(2t)+\frac{1}{2}\cos u\right)-\frac{(t-h)^2}{2d\beta}\right).
\end{equation}
Thus, we will get  the asymptotic expansion of $W(x)$ via Laplace's method for multiple integrals. Fulks and Sather introduced this method in \cite{FulksSatherMR0138945}. However, \cite{FulksSatherMR0138945} did not include the critical case $\beta=\beta_c$ and $h=0$, and their other result is not useful for us to calculate the coefficients in the asymptotic expansion of $W(x)$. Therefore, we decide to get the asymptotic expansion of $W(x)$ by ourselves. In the following, we provide the details of the asymptotic expansion of $W(x)$ for the critical case $h=0$ and $\beta=\beta_{c}$, because this case is not covered by the paper \cite{FulksSatherMR0138945} and we also think that  it is the most technical. For the other cases, we only briefly present the proof.

For the critical case, $\beta=\beta_c$ and $h=0$, we have that for all $t\in \mathbb{R}$ and $u\in (-\pi,\pi)$,
$$\psi(t,u)\leq \psi(t,0)=\ln\cosh t-\frac{t^2}{4d\beta}=\varphi(t),$$
where $\varphi(t)$ is defined in \eqref{eq: defn varphi t}. Hence, by Lemma~\ref{lem: the maximum point of varphi}, $(0,0)$ is the unique maximum point of $\psi(t,u)$. Thus, we truncate $W(x)$ as follows: for $x\in (0,0.01)$,
\begin{equation}\label{eq: truncation of W(x) critical}
W(x)=\widetilde{W}^c(x)+\widetilde{R}^c(x),
\end{equation}
where
\begin{equation}\label{eq: even tilde W critical}
    \widetilde{W}^c(x)=\frac{1}{(2\pi)^{\frac{3}{2}}x}\int_{-\sqrt{-x\ln x}}^{\sqrt{-x\ln x}}\int_{x\ln x}^{-x\ln x}e^{\psi(t,u)/x^2}\,\mathrm{d}t\mathrm{d}u.
\end{equation}
For sufficienty small $t$ and $u$, the Taylor expansion of $\psi(t,u)$ is
\begin{equation}\label{eq: psi taylor expansion critical}
\psi(t,u)=-\frac{1}{8}u^2-\frac{1}{12}t^4+\sum_{\substack{p,q\geq 0\\p+2q\geq 3}}\alpha_{2p,2q}t^{2p}u^{2q},
\end{equation}
where for $p+2q\geq 3$,
\begin{align*}
\alpha_{2p,2q}&=2^{2p-1}(-1)^q\\
&\quad\times\sum^{p+q}_{n=1}\frac{(-1)^{n-1}}{n2^n}
\sum_{\substack{m_i\geq 1, 1\leq i\leq n\\ \sum^n_{i=1}m_i=p+q}}\prod^n_{i=1}\frac{1}{(2m_i)!}\sum_{\substack{j_i=m_i\,\,or\,\,0\\ k_i=m_i-j_i}}I_{(p,q)}(\sum^n_{i=1}j_i, \sum^n_{i=1}k_i).
\end{align*}
Hence, there exists $\delta>0$ and $\kappa>0$ such that for $|t|\leq\delta$ and $|u|\leq\delta$, we have that
\[\psi(t,u)\leq-\kappa u^2-\kappa t^4\leq-\kappa u^2\]
and that for $|t|\leq\delta$ and $|u|\in[\delta,\pi]$,
\[\psi(t,u)\leq \psi(t,\delta)\leq -\kappa\delta^2.\]
Hence, for sufficiently small $x$, we have that
\begin{align*}
\widetilde{R}^c(x)&=\frac{1}{(2\pi)^{\frac{3}{2}}x}\iint_{|t|>\sqrt{-x\ln x}, |u|\leq \pi}e^{\psi(t,u)/x^2}\,\mathrm{d}t\mathrm{d}u\notag\\
&\quad +\frac{1}{(2\pi)^{\frac{3}{2}}x}\iint_{|t|<\sqrt{-x\ln x},-x\ln x<|u|<\pi}e^{\psi(t,u)/x^2}\,\mathrm{d}t\mathrm{d}u\\
&\leq \frac{1}{(2\pi)^{\frac{1}{2}}x}\int_{|t|>\sqrt{-x\ln x}}e^{\varphi(t)/x^2}\,\mathrm{d}t\\
&\quad+\frac{1}{(2\pi)^{\frac{3}{2}}x}\iint_{|t|\leq\delta,-x\ln x<|u|<\delta}e^{-\kappa u^2/x^2}\,\mathrm{d}t\mathrm{d}u\\
&\quad+\frac{1}{(2\pi)^{\frac{3}{2}}x}\iint_{|t|\leq\delta,\delta\leq|u|\leq\pi}e^{-\kappa\delta^2/x^2}\,\mathrm{d}t\mathrm{d}u.
\end{align*}
By \eqref{eq: truncation of Z(x) critical 02} and \eqref{eq: upper bounds for tilde R from Z}, there exist $0<c,C<+\infty$ such that the first term is bounded from above by $Ce^{-c(\ln x)^2}$. By the tail bound for Gaussian distribution, there exist $0<c,C<+\infty$ such that the second term is bounded from above by $Ce^{-c(\ln x)^2}$. And there exists $c>0$ such that the third term is much smaller than $e^{-c(\ln x)^2}$ as $x\to 0^+$. Hence, there exist $0<c,C<+\infty$ such that the third term is bounded from above by $Ce^{-c(\ln x)^2}$. Hence, there exist $0<c,C<+\infty$ such that
\begin{equation}\label{eq: upper bounds for tilde R c}
\widetilde{R}^c(x)\leq Ce^{-c(\ln x)^2}.
\end{equation}
Hence, $\widetilde{R}^c(x)$ is much smaller than  $x^{M}$  for any $M\geq 1$.

For the main part $\widetilde{W}^c(x)$, by \eqref{eq: even tilde W critical} and \eqref{eq: psi taylor expansion critical}, we have that
\begin{align}\label{eq: critical expansion of tilde W via integral}
\widetilde{W}^c(x)&=\frac{1}{(2\pi)^{\frac{3}{2}}x}\int_{-\sqrt{-x\ln x}}^{\sqrt{-x\ln x}}\int_{x\ln x}^{-x\ln x}e^{(-\frac{1}{8}u^2-\frac{1}{12}t^4+\sum_{\substack{p,q\geq 0\\p+2q\geq 3}}\alpha_{2p,2q}t^{2p}u^{2q})/x^2}\,\mathrm{d}t\mathrm{d}u\notag\\
&\xlongequal{t=s\sqrt{x}, u=vx} \frac{\sqrt{x}}{(2\pi)^{\frac{3}{2}}}\int_{-\sqrt{-\ln x}}^{\sqrt{-\ln x}}\int_{\ln x}^{-\ln x}e^{-\frac{1}{12}s^4-\frac{1}{8}v^2}\notag\\
&\quad\quad\quad\quad\quad\quad\times e^{\sum_{p,q\geq 0,p+2q\geq 3}\alpha_{2p,2q}s^{2p}v^{2q}x^{p+2q-2}}\,\mathrm{d}s\mathrm{d}v\notag\\
&= \frac{\sqrt{x}}{(2\pi)^{\frac{3}{2}}}\int_{-\sqrt{-\ln x}}^{\sqrt{-\ln x}}\int_{\ln x}^{-\ln x}e^{-\frac{1}{12}s^4-\frac{1}{8}v^2}\notag\\
&\quad\quad\quad\quad\quad\quad\times\left(1+\sum_{p=1}^{M}\beta^c_{p}(s,v)x^{p}+\widetilde{r}^c_{M}(s,v,x)\right)\,\mathrm{d}s\mathrm{d}v,
\end{align}
where
\begin{equation}\label{eq: beta s v critical}
    \beta^c_p(s,v)=\sum_{k=1}^{p}\frac{1}{k!}\sum_{\substack{p_i+2q_i\geq 3, p_i\geq 0,q_i\geq 0, 1\leq i\leq k,\\\sum^k_{i=1}p_i+2\sum^k_{i=1}q_i-2k=p}}
\prod_{j=1}^{k}\alpha_{2p_j,2q_j}s^{2\sum_{j=1}^{k}p_j}v^{2\sum_{j=1}^{k}q_j},
\end{equation}
and there exists $C=C(M)<+\infty$ such that for $|s|<\sqrt{-\ln x}$ and $|v|<-\ln x$, we have that
\[|\widetilde{r}^c_{M}(s,v,x)|\leq Cx^{M+1}(-\ln x)^{C}=o(x^{M}).\]
Note that for non-negative integers $p$ and $q$,
\begin{align}\label{eq: truncated e-s4s2pv2q integral}
\int_{-\sqrt{-\ln x}}^{\sqrt{-\ln x}}\int_{\ln x}^{-\ln x}&e^{-\frac{1}{12}s^{4}-\frac{1}{8}v^{2}}s^{2p}v^{2q}\,\mathrm{d}s\mathrm{d}v\notag\\ &=\int_{-\infty}^{+\infty}\int_{-\infty}^{+\infty}e^{-\frac{1}{12}s^{4}-\frac{1}{8}v^{2}}s^{2p}v^{2q}\,\mathrm{d}s\mathrm{d}v+o(x^{M})\notag\\
&=3^{\frac{p}{2}+\frac{1}{4}}2^{p+3q+1}\Gamma\left(\frac{p}{2}+\frac{1}{4}\right)\Gamma\left(q+\frac{1}{2}\right)+o(x^{M}).
\end{align}
Combining \eqref{eq: truncation of W(x) critical}, \eqref{eq: upper bounds for tilde R c}, \eqref{eq: critical expansion of tilde W via integral}, \eqref{eq: beta s v critical} and \eqref{eq: truncated e-s4s2pv2q integral}, for $M\geq 1$, we have that
\begin{equation*}
W(x)=\sum_{p=0}^{M}\gamma^c_px^{p+\frac{1}{2}}+o(x^{M+\frac{1}{2}}),
\end{equation*}
where $\gamma^c_0=\frac{1}{\pi}\left(\frac{3}{4}\right)^{\frac{1}{4}}\Gamma(\frac{1}{4})$ and for $p\geq 1$,

\begin{align}\label{eq: critical gamma p}
\gamma^c_p&=\frac{1}{(2\pi)^{3/2}}\sum_{k=1}^{p}\frac{1}{k!}\sum_{\substack{p_i+2q_i\geq 3, p_i\geq 0,q_i\geq 0, 1\leq i\leq k,\\\sum^k_{i=1}p_i+2\sum^k_{i=1}q_i-2k=p}}\prod_{j=1}^{k}\alpha_{2p_j,2q_j}\notag\\
&\quad\times 3^{\frac{1}{2}\sum^k_{i=1}p_i+\frac{1}{4}}2^{\sum^k_{i=1}p_i+3\sum^k_{i=1}q_i+1}\Gamma\left(\frac{1}{2}\sum_{j=1}^{k}p_j+\frac{1}{4}\right)\Gamma\left(\sum_{j=1}^{k}q_j+\frac{1}{2}\right).
\end{align}
By \eqref{eq: psi taylor expansion critical}, we have that $\alpha_{2,2}=1/8$ and $\alpha_{6,0}=1/45$. Then, by \eqref{eq: critical gamma p}, we have that $\gamma^c_{1}=\frac{7\cdot 3^{3/4}\Gamma(\frac{3}{4})}{5\cdot 2^{\frac{1}{2}}\pi}$. Thus by \eqref{eq: O even via W}, Proposition~\ref{prop: O asymptotics even} (2) holds.

For the case $h=0$, $\beta<\beta_c$ or the case $h>0$, we have that for all $t\in \mathbb{R}$ and $u\in (-\pi,\pi)$,
$$\psi(t,u)\leq \psi(t,0)=\ln\cosh t-\frac{(t-h)^2}{4d\beta}=\varphi(t),$$
where $\varphi(t)$ is defined in \eqref{eq: defn varphi t}. Hence, by Lemma~\ref{lem: the maximum point of varphi}, $(t_*,0)$ is the unique maximum point of $\psi(t,u)$. In particular, for the case $h=0, \beta<\beta_c$, $t_*=0$; and for the case $h\neq 0$, $t_*>0$. Thus, we truncate $W(x)$ as follows: for $x\in (0,0.01)$,
\begin{equation}\label{eq: truncation of W(x)}
W(x)=\widetilde{W}^{\S}(x)+\widetilde{R}^{\S}(x),
\end{equation}
where
\begin{equation*}\label{eq: even tilde W tow cases}
 \widetilde{W}^{\S}(x)=\frac{1}{(2\pi)^{\frac{3}{2}}\sqrt{2d\beta}x}\int_{t_*+x\ln x}^{t_*-x\ln x}\int_{x\ln x}^{-x\ln x}e^{\psi(t,u)/x^2}\,\mathrm{d}t\mathrm{d}u.
\end{equation*}
The Taylor's expansion of the function $\psi(t,u)$ at the maximum points $(t_*,0)$ is
\begin{equation}\label{eq: Taylor expansion of psi at t star 0}
\psi(t,u)=\psi(t_{*},0)+\sum_{p+q\geq 2}\alpha_{p,q}(t-t_{*})^pu^q,
\end{equation}
where
\begin{equation}\label{eq: taylor coe high temperature and with field}
2\alpha_{2,0}=\frac{1}{\cosh^2 t_{*}}-\frac{1}{2d\beta}<0,\,\,\alpha_{1,1}=0,\,\,2\alpha_{0,2}=-\frac{1}{4}\frac{1}{\cosh^2 t_{*}}.
\end{equation}
Similarly, there exist $0<c,C<+\infty$ such that
\begin{equation}\label{eq: upper bounds for tilde R S}
\widetilde{R}^{\S}(x)\leq Ce^{-c(\ln x)^2}.
\end{equation}
And by \eqref{eq: Taylor expansion of psi at t star 0} we have
\begin{align}\label{eq: expansion of tilde W S via integral}
\widetilde{W}^{\S}(x)&=\frac{1}{(2\pi)^{\frac{3}{2}}\sqrt{2d\beta}x}\int_{t_*+x\ln x}^{t_*-x\ln x}\int_{x\ln x}^{-x\ln x}e^{\frac{1}{x^2}(\psi(t_*,0)+\sum_{p+q\geq 2}\alpha_{p,q}(t-t_*)^pu^q)}\,\mathrm{d}t\mathrm{d}u\notag\\
&\xlongequal{t=t_*+sx, u=vx} \frac{xe^{\psi(t_*,0)/x^2}}{(2\pi)^{\frac{3}{2}}\sqrt{2d\beta}}\int_{\ln x}^{-\ln x}\int_{\ln x}^{-\ln x}e^{\alpha_{2,0}s^2+\alpha_{0,2}v^2}\notag\\
&\quad\quad\quad\quad\quad\quad\times e^{\sum_{p,q\geq 0,p+q\geq 3}\alpha_{p,q}s^{p}v^{q}x^{p+q-2}}\,\mathrm{d}s\mathrm{d}v\notag\\
 &= \frac{xe^{\psi(t_*,0)/x^2}}{(2\pi)^{\frac{3}{2}}\sqrt{2d\beta}}\int_{\ln x}^{-\ln x}\int_{\ln x}^{-\ln x}e^{\alpha_{2,0}s^2+\alpha_{0,2}v^2}\notag\\
 &\quad\quad\quad\quad\quad\quad\times\left(1+\sum_{p=1}^{M}\beta_{p}^{\S}(s,v)x^{p}+\widetilde{r}_{M}^{\S}(s,v,x)\right)\,\mathrm{d}s\mathrm{d}v,
\end{align}
 where
 \begin{equation}\label{eq: beta S s v}
     \beta_p^{\S}(s,v)=\sum_{k=1}^{p}\frac{1}{k!}\sum_{\substack{p_i+q_i\geq 3, p_i\geq 0,q_i\geq 0, 1\leq i\leq k,\\\sum^k_{i=1}p_i+\sum^k_{i=1}q_i-2k=p}}
 \prod_{j=1}^{k}\alpha_{p_j,q_j}s^{\sum_{j=1}^{k}p_j}v^{\sum_{j=1}^{k}q_j}
 \end{equation}
 and there exists $C=C(M)<+\infty$ such that for $|s|<-\ln x$ and $|v|<-\ln x$, it holds that
 \[\widetilde{r}_{M}^{\S}(s,v,x)=o(x^{M}).\]
Note that
\begin{align}\label{eq: truncated integral even}
&\int_{\ln x}^{-\ln x}\int_{\ln x}^{-\ln x} e^{\alpha_{2,0}s^2+\alpha_{0,2}v^{2}}s^{p}v^{q}\,\mathrm{d}s\mathrm{d}v\notag\\ &=\int_{-\infty}^{+\infty}\int_{-\infty}^{+\infty}e^{\alpha_{2,0}s^2+\alpha_{0,2}v^{2}}s^{p}v^{q}\,\mathrm{d}s\mathrm{d}v+o(x^{M})\notag\\
&=1_{p\text{ and }q\text{ are even}}\times(-\alpha_{2,0})^{-\frac{p+1}{2}}(-\alpha_{0,2})^{-\frac{q+1}{2}}\Gamma\left(\frac{p+1}{2}\right)\Gamma\left(\frac{q+1}{2}\right)+o(x^{M}).
\end{align}
So, by \eqref{eq: truncation of W(x)}, \eqref{eq: upper bounds for tilde R S}, \eqref{eq: expansion of tilde W S via integral}, \eqref{eq: beta S s v} and \eqref{eq: truncated integral even}, for positive integer $M\geq 1$, we have that
\begin{equation*}\label{eq: W two case}
    W(x)=e^{\frac{\psi(t_*,0)}{x^2}}\left[\sum^M_{p=0}\gamma_p^{\S}x^{2p+1}+o(x^{2M+1})\right],
\end{equation*}
where $\psi(t_*,0)=\varphi(t_*)$,
\begin{align*}
\gamma_0^{\S}&=\frac{1}{(2\pi)^{\frac{3}{2}}\sqrt{2d\beta}}(-\alpha_{2,0})^{-1/2}(-\alpha_{0,2})^{-1/2}\Gamma^2\left(\frac{1}{2}\right)\notag\\
&=\sqrt{\frac{2}{\pi}}(2d\beta)^{-1/2}\left(\frac{1}{2d\beta}-\frac{1}{\cosh^2t_*}\right)^{-1/2}\cosh t_*,
\end{align*}
\begin{align}\label{eq: gamma p two cases}
\gamma_p^{\S}&=\frac{1}{(2\pi)^{\frac{3}{2}}\sqrt{2d\beta}}\sum^{2p}_{k=1}\frac{1}{k!}\sum_{\substack{p_i+q_i\geq 3, p_i\geq 0,q_i\geq 0, 1\leq i\leq k,\\\sum^k_{i=1}(p_i+q_i)=2(k+p),\\ \sum^k_{i=1}p_i\text{ even}, \sum^k_{i=1}q_i\text{ even}}}\prod_{j=1}^{k}\alpha_{p_j,q_j}\notag\\
&\qquad \times (-\alpha_{2,0})^{-\frac{1}{2}(\sum^k_{i=1}p_i+1)}(-\alpha_{0,2})^{-\frac{1}{2}(\sum^k_{i=1}q_i+1)}\notag\\
&\qquad \times \Gamma\left((\sum^k_{i=1}p_i+1)/2\right)\Gamma\left((\sum^k_{i=1}q_i+1)/2\right), \quad p\geq 1.
\end{align}
Thus, by \eqref{eq: O even via W} Proposition~\ref{prop: O asymptotics even} (1) and (4) hold.

For the case $h=0, \beta>\beta_c$, we still have that for all $t\in \mathbb{R}$ and $u\in (-\pi,\pi)$,
$$\psi(t,u)\leq \psi(t,0)=\ln\cosh t-\frac{t^2}{4d\beta}=\varphi(t),$$
where $\varphi(t)$ is defined in \eqref{eq: defn varphi t}. Hence, by Lemma~\ref{lem: the maximum point of varphi}, $\psi(t,u)$ has two maximum point $(t_*,0)$ and $(-t_*,0)$. And for fixed $u\in(-\pi,\pi)$, $\psi(t,u)$ is a even function for $t$.  Hence, for  $x\in (0,0.01)$, we truncate $W(x)$ as follows:
\begin{equation}\label{eq: truncation of W(x) sup}
W(x)=\frac{1}{2^{1/2}\pi^{3/2}\sqrt{2d\beta}x}\int^{+\infty}_0\int^{\pi}_{-\pi}e^{\psi(t,u)/x^2}\,\mathrm{dt}\mathrm{du}=\widetilde{W}^{>}(x)+\widetilde{R}^{>}(x),
\end{equation}
where
\[\widetilde{W}^{>}(x)=\frac{1}{2^{\frac{1}{2}}\pi^{\frac{3}{2}}\sqrt{2d\beta}x}\int_{t_*+x\ln x}^{t_*-x\ln x}\int_{x\ln x}^{-x\ln x}e^{\psi(t,u)/x^2}\,\mathrm{d}t\mathrm{d}u.\]
The Taylor's expansion of the function $\psi(t,u)$ at the maximum points $(t_*,0)$ is the same as in \eqref{eq: Taylor expansion of psi at t star 0} and \eqref{eq: taylor coe high temperature and with field}.
So similarly, there exists $c>0$ and $C<+\infty$ such that
\begin{equation}\label{eq: upper bounds for tilde R >}
\widetilde{R}^{>}(x)\leq Ce^{-c(\ln x)^2}.
\end{equation}
And
\begin{align}\label{eq: expansion of tilde W via integral sup}
\widetilde{W}^>(x)&=\frac{1}{2^{\frac{1}{2}}\pi^{\frac{3}{2}}\sqrt{2d\beta}x}\int_{t_*+x\ln x}^{t_*-x\ln x}\int_{x\ln x}^{-x\ln x}e^{\frac{1}{x^2}(\psi(t_*,0)+\sum_{p+q\geq 2}\alpha_{p,q}(t-t_*)^pu^q)}\,\mathrm{d}t\mathrm{d}u\notag\\
&\xlongequal{t=t_*+sx, u=vx} \frac{xe^{\psi(t_*,0)/x^2}}{2^{\frac{1}{2}}\pi^{\frac{3}{2}}\sqrt{2d\beta}}\int_{\ln x}^{-\ln x}\int_{\ln x}^{-\ln x}e^{\alpha_{2,0}s^2+\alpha_{0,2}v^2}\notag\\
&\quad\quad\quad\quad\quad\quad\times e^{\sum_{p,q\geq 0,p+q\geq 3}\alpha_{p,q}s^{p}v^{q}x^{p+q-2}}\,\mathrm{d}s\mathrm{d}v\notag\\
&= \frac{xe^{\psi(t_*,0)/x^2}}{2^{\frac{1}{2}}\pi^{\frac{3}{2}}\sqrt{2d\beta}}\int_{\ln x}^{-\ln x}\int_{\ln x}^{-\ln x}e^{\alpha_{2,0}s^2+\alpha_{0,2}v^2}\notag\\
&\quad\quad\quad\quad\quad\quad\times
\left(1+\sum_{p=1}^{M}\beta_{p}^{>}(s,v)x^{p}+\widetilde{r}_{M}^{>}(s,v,x)\right)\,\mathrm{d}s\mathrm{d}v,
\end{align}
where
\begin{equation}\label{eq: beta > s v}
     \beta_p^{>}(s,v)=\sum_{k=1}^{p}\frac{1}{k!}\sum_{\substack{p_i+q_i\geq 3, p_i\geq 0,q_i\geq 0, 1\leq i\leq k,\\\sum^k_{i=1}p_i+\sum^k_{i=1}q_i-2k=p}}
 \prod_{j=1}^{k}\alpha_{p_j,q_j}s^{\sum_{j=1}^{k}p_j}v^{\sum_{j=1}^{k}q_j}
 \end{equation}
 and there exists $C=C(M)<+\infty$ such that for $|s|<-\ln x$ and $|v|<-\ln x$, we have that
 \[\widetilde{r}_{M}^{>}(s,v,x)=o(x^{M}).\]
Then by \eqref{eq: truncated integral even},  \eqref{eq: truncation of W(x) sup}, \eqref{eq: upper bounds for tilde R >}, \eqref{eq: expansion of tilde W via integral sup} and \eqref{eq: beta > s v}, for positive integer $M\geq 1$, we have that
\begin{equation*}
    W(x)=e^{\frac{\psi(t_*,0)}{x^2}}\left[\sum^M_{p=0}\gamma^>_px^{2p+1}+o(x^{2M+1})\right],
\end{equation*}
where $\psi(t_*,0)=\varphi(t_*)$,
\begin{align*}
\gamma^>_0&=\frac{1}{2^{\frac{1}{2}}\pi^{\frac{3}{2}}\sqrt{2d\beta}}(-\alpha_{2,0})^{-1/2}(-\alpha_{0,2})^{-1/2}\Gamma^2(\frac{1}{2})\\ &=2\sqrt{\frac{2}{\pi}}(2d\beta)^{-1/2}(\frac{1}{2d\beta}-\frac{1}{\cosh^2t_*})^{-1/2}\cosh t_*,
\end{align*}
\begin{align}\label{eq: gamma p low teperature}
\gamma_p^{>}&=\frac{1}{2^{\frac{1}{2}}\pi^{\frac{3}{2}}\sqrt{2d\beta}}\sum^{2p}_{k=1}\frac{1}{k!}\sum_{\substack{p_i+q_i\geq 3, p_i\geq 0,q_i\geq 0, 1\leq i\leq k,\\\sum^k_{i=1}(p_i+q_i)=2(k+p),\\ \sum^k_{i=1}p_i\text{ even}, \sum^k_{i=1}q_i\text{ even}}}\prod_{j=1}^{k}\alpha_{p_j,q_j}\notag\\
&\qquad \times (-\alpha_{2,0})^{-\frac{1}{2}(\sum^k_{i=1}p_i+1)}(-\alpha_{0,2})^{-\frac{1}{2}(\sum^k_{i=1}q_i+1)}\notag\\
&\qquad \times \Gamma\left(\frac{1}{2}(\sum^k_{i=1}p_i+1)\right)\Gamma\left(\frac{1}{2}(\sum^k_{i=1}q_i+1)\right), \,\, p\geq 1.
\end{align}
Thus, by \eqref{eq: O even via W} Proposition~\ref{prop: O asymptotics even} (3) holds.
\end{proof}

Therefore, by \eqref{eq: Qn = O/Z}, Theorem~\ref{thm: normalizing constant expansion} and Proposition~\ref{prop: O asymptotics even}, we have the following corollary:

\begin{corollary}\label{cor: Q asymptotics even}
Let $n$ be even number. As $n\to+\infty$, the asymptotic expansion of $Q_n$ are as follows:
\begin{enumerate}[(1)]
  \item For $h=0$ and $\beta<\beta_c$, we have that
  \begin{equation*}
  Q_n=\sum^{M}_{p=0}H_p^{<}n^{-p-\frac{1}{2}}+o(n^{-M-\frac{1}{2}}),
  \end{equation*}
  where $H^<_0=\sqrt{\frac{2}{\pi}}$.
  \item For $h=0$ and $\beta=\beta_c$, we have that
  \begin{equation*}
  Q_n=\sum^{M}_{p=0}H_p^{c}n^{-\frac{p}{2}-\frac{1}{2}}+o(n^{-\frac{M}{2}-\frac{1}{2}}),
  \end{equation*}
  where $H_0^{c}=\sqrt{\frac{2}{\pi}}$ and $H_1^{c}=\frac{2\cdot 3^{\frac{1}{2}}\cdot \pi^{\frac{1}{2}}}{(\Gamma(\frac{1}{4}))^2}$.
  \item For $h=0$ and $\beta>\beta_c$, we have that
  \begin{equation*}
   Q_n=\sum^{M}_{p=0}H_p^{>}n^{-p-\frac{1}{2}}+o(n^{-M-\frac{1}{2}}),
  \end{equation*}
  where $H^>_0=\sqrt{\frac{2}{\pi}}\cosh t_*$ and $z_{*}=\frac{t_*}{2d\beta}$ is the maximal solution of the mean-field equation $\tanh\left(\frac{\beta}{\beta_c} z\right)=z$.
  \item For $h\neq 0$, we have that
   \begin{equation*}
   Q_n=\sum^{M}_{p=0}H_p^{h\neq 0}n^{-p-\frac{1}{2}}+o(n^{-M-\frac{1}{2}}),
  \end{equation*}
  where $H^{h\neq 0}_0=\sqrt{\frac{2}{\pi}}\cosh t_*$ and $z_{*}=\frac{t_*-|h|}{2d\beta}$ is the maximal solution of the mean-field equation
  $$\tanh\left(\frac{\beta}{\beta_c} z+|h|\right)=z.$$
\end{enumerate}
\end{corollary}
From Corollary~\ref{cor: Q asymptotics even}, we know that Theorem \ref{thm: Littlewood-Offord problem asymptotics} holds when $n$ is even.

\subsubsection{For odd \texorpdfstring{$n$}{n}}\label{subsubsection: odd n}
Let $n$ be an odd integer. Then, $n-1$ is an even integer. Let
\begin{equation*}\label{eq: defn Pn odd}
 P_n=P\left(\sum_{i=1}^{(n-1)/2}\varepsilon_i-\sum_{i=(n+1)/2}^{n}\varepsilon_i=1\right).
 \end{equation*}
 Then by \eqref{eq: odd n upper lower bounds on Qn}, we get
 \begin{equation}\label{eq: Pn and Qn odd}
     P_n\leq Q_n\leq Q_{n-1}.
 \end{equation}
 Because we have get the asymptotic expansion of $Q_{n-1}$ in Corollary~\ref{cor: Q asymptotics even}, it suffices to study the asymptotic expansion of $P_n$ in the following.

We write
\begin{equation}\label{eq: defn Pn}
P_n=\frac{O^{odd}_{d,\beta,h}}{Z_{d,\beta,h}},
\end{equation}
where
\begin{equation}\label{eq: odd O d beta h}
O^{odd}_{d,\beta,h} = \sum_{\sigma\in\{-1,1\}^n:\sum_{i=1}^{(n-1)/2}\sigma_i-\sum_{i=(n+1)/2}^{n}\sigma_i=1} \exp\left\{\frac{d\beta}{n}(\sum_{j=1}^{n}\sigma_j)^2+h\sum_{j=1}^{n}\sigma_j\right\}.
\end{equation}
Because the asymptotic expansion of $Z_{d,\beta,h}$ is given in Section~\ref{sect: asymptotic expansion of Z as a series in n}, it remains to calculate the asymptotic expansion of $O^{odd}_{d,\beta,h}$. The strategy is similar to the calculation of the expansion of $O_{d,\beta,h}$. So we will not provide the full detail. Instead, we point out the key argument.

\begin{lemma}\label{lem: expectation representation of O plus d beta h}
Let $n$ be a positive odd integer. Then, we have that
\begin{equation*}
   O^{odd}_{d,\beta,h}=2^{\frac{n-1}{2}}E\left[\begin{array}{l}
   \left(\cosh\left(2\sqrt{\frac{2d\beta}{n}}Y+2h\right)+\cos U\right)^{\frac{n-1}{2}}\\
   \times\left(e^{-(\sqrt{\frac{2d\beta}{n}}Y+h)}+e^{\sqrt{\frac{2d\beta}{n}}Y+h}\cos U\right)
   \end{array}\right]
\end{equation*}
where $Y$ is a standard Gaussian random variable, $U$ follows the uniform distribution on $(-\pi,\pi)$,  $Y$ and $U$ are independent.
\end{lemma}
\begin{proof}
Let $Y$ be a standard Gaussian random variable. Let $T_n=\sum^{\frac{n-1}{2}}_{i=1}\varepsilon_i-\sum^n_{i=\frac{n+1}{2}}\varepsilon_i$. Hence, by \eqref{eq: varepsilon conditional disitribution with Y} and \eqref{eq: odd O d beta h}, we have that
\begin{align}\label{eq: calculate odd O d beta h}
O^{odd}_{d,\beta,h}&=Z_{d,\beta,h}P(T_n=1)\notag\\
&=\int_{\mathbb{R}}P(T_n=1|Y=y)2^n\cosh^n\left(\sqrt{\frac{2d\beta}{n}}y+h\right)\frac{1}{\sqrt{2\pi}}e^{-\frac{y^2}{2}}\,\mathrm{d}y.
\end{align}
Thus by \eqref{eq: varepsilon conditional distribution}, we get that
\begin{align*}
 \varphi_{T_n|Y}(t|y)&:=E(e^{itT_n}|Y=y)\notag\\
 &=|E(e^{it\varepsilon_1}|Y=y)|^{n-1}E(e^{-it\varepsilon_1}|Y=y)\notag\\
 &=\frac{1}{2^{\frac{n-1}{2}}\left(\cosh\left(\sqrt{\frac{2d\beta}{n}}y+h\right)\right)^{n-1}}\notag\\
 &\qquad \times\left(\cosh\left(2\left(\sqrt{\frac{2d\beta}{n}}y+h\right)\right)+\cos(2t)\right)^{\frac{n-1}{2}}\notag\\
 &\qquad\times\left(\cos t-i\sin t\tanh\left(\sqrt{\frac{2d\beta}{n}}y+h\right)\right).
\end{align*}
As $T_n$ is integer-valued, we have that
\begin{align}\label{eq: odd T_n = 1 conditional prob} P(T_n=1|Y=y)\notag&=\frac{1}{2\pi}\int_{-\pi}^{\pi}e^{-it}\varphi_{T_n|Y}(t|y)(t)\,\mathrm{d}t\notag\\
&=\frac{1}{2\pi}\int_{-\pi}^{\pi}\mathrm{Re}(e^{-it}\varphi_{T_n|Y}(t|y))\,\mathrm{d}t\notag\\
&=\frac{1}{2\pi\cdot 2^{\frac{n+1}{2}}\left(\cosh\left(\sqrt{\frac{2d\beta}{n}}y+h\right)\right)^{n}}\notag\\
&\quad\times \int^{\pi}_{-\pi}\left(\cosh\left(2\left(\sqrt{\frac{2d\beta}{n}}y+h\right)\right)+\cos(2t)\right)^{\frac{n-1}{2}}\notag\\
&\quad\quad\quad\quad\left(e^{-(\sqrt{\frac{2d\beta}{n}}y+h)}+e^{\sqrt{\frac{2d\beta}{n}}y+h}\cos(2t)\right)\,\mathrm{d}t.
\end{align}
Hence, by \eqref{eq: calculate odd O d beta h} and \eqref{eq: odd T_n = 1 conditional prob}, we get that
\begin{align*}
   O^{odd}_{d,\beta,h}&=\int_{\mathbb{R}}\int^{\pi}_{-\pi}\left(\cosh\left(2\left(\sqrt{\frac{2d\beta}{n}}y+h\right)\right)+\cos(2t)\right)^{\frac{n-1}{2}}\notag\\
   &\qquad\times\left(e^{-(\sqrt{\frac{2d\beta}{n}}y+h)}+e^{\sqrt{\frac{2d\beta}{n}}y+h}\cos(2t)\right) \frac{2^{(n-1)/2}}{(2\pi)^{3/2}}e^{-\frac{y^2}{2}}\,\mathrm{d}y\mathrm{d}t\\
&\xlongequal{2t=u}\int_{\mathbb{R}}\int^{2\pi}_{-2\pi}\left(\cosh\left(2\left(\sqrt{\frac{2d\beta}{n}}y+h\right)\right)+\cos u\right)^{\frac{n-1}{2}}\notag\\
   &\qquad\times\left(e^{-(\sqrt{\frac{2d\beta}{n}}y+h)}+e^{\sqrt{\frac{2d\beta}{n}}y+h}\cos u\right)\frac{2^{(n-3)/2}}{(2\pi)^{3/2}}e^{-\frac{y^2}{2}}\,\mathrm{d}y\mathrm{d}u\\
   &=2^{(n-1)/2}\int_{\mathbb{R}}\int^{\pi}_{-\pi}\left(\cosh\left(2\left(\sqrt{\frac{2d\beta}{n}}y+h\right)\right)+\cos u\right)^{\frac{n-1}{2}}\notag\\
   &\qquad\times\left(e^{-(\sqrt{\frac{2d\beta}{n}}y+h)}+e^{\sqrt{\frac{2d\beta}{n}}y+h}\cos u\right)\frac{1}{(2\pi)^{3/2}}e^{-\frac{y^2}{2}}\,\mathrm{d}y\mathrm{d}u\\
   &=2^{\frac{n-1}{2}}E\left[
   \left(\cosh\left(2\left(\sqrt{\frac{2d\beta}{n}}Y+h\right)\right)+\cos U\right)^{\frac{n-1}{2}}\right.\\
   &\quad \times \left.\left(e^{-(\sqrt{\frac{2d\beta}{n}}Y+h)}+e^{\sqrt{\frac{2d\beta}{n}}Y+h}\cos U\right)
  \right],
\end{align*}
where $U$ follows the uniform distribution on $(-\pi,\pi)$ and it is independent of $Y$.
\end{proof}

  \begin{proposition}\label{prop: Pn expansion odd}
 Let $n$ be an odd integer. As $n\to+\infty$, the asymptotic expansion of $P_n$ is as follows:
\begin{enumerate}[(1)]
  \item For $h=0$ and $\beta<\beta_c$, we have that
  \begin{equation*}
  P_n=\sqrt{\frac{2}{\pi}}n^{-1/2}+O(n^{-3/2}).
  \end{equation*}
   \item For $h=0$ and $\beta=\beta_c$, we have that
  \begin{equation*}
  P_n=\sqrt{\frac{2}{\pi}}n^{-\frac{1}{2}}+\frac{2\cdot 3^{\frac{1}{2}}\cdot\pi^{\frac{1}{2}}}{(\Gamma(1/4))^2}n^{-1}+o(n^{-1}).
  \end{equation*}
  \item For $h=0$ and $\beta>\beta_c$, we have that
  \begin{equation*}
   P_n=\sqrt{\frac{2}{\pi}}\cosh(t_*)n^{-1/2}+O(n^{-3/2}),
  \end{equation*}
  where  $z_{*}=\frac{t_*}{2d\beta}$ is the maximal solution of the mean-field equation $\tanh\left(\frac{\beta}{\beta_c} z\right)=z.$
  \item For $h\neq 0$, we have that
   \begin{equation*}
   P_n=\sqrt{\frac{2}{\pi}}\cosh(t_*)n^{-1/2}+O(n^{-3/2}),
  \end{equation*}
  where $z_{*}=(t_{*}-|h|)/(2d\beta)$ is the maximal solution of the mean-field equation \[\tanh\left(\frac{\beta}{\beta_c} z+|h|\right)=z.\]
\end{enumerate}
 \end{proposition}
\begin{proof}
The proof is very similar to that of Proposition~\ref{prop: O asymptotics even}, so we only point out the key argument.

For $x\neq 0$, we define
\begin{equation*}
W^{odd}(x)=E\left[\begin{array}{l}
\exp\{\frac{1}{2}\ln\left(\frac{1}{2}\cosh\left(2\sqrt{2d\beta}Yx+2h\right)+\frac{1}{2}\cos U\right)/x^2\}\\
\times\left(e^{-\sqrt{2d\beta}Yx-h}+e^{\sqrt{2d\beta}Yx+h}\cos U\right)\\
\times\left(2\cosh\left(2\sqrt{2d\beta}Yx+2h\right)+2\cos U\right)^{-\frac{1}{2}}
\end{array}\right].
\end{equation*}
Then, we have that
\begin{equation}\label{eq: O odd via W odd}
O_{d,\beta,h}^{odd}=2^nW^{odd}(1/\sqrt{n}).
\end{equation}
For $x\neq 0$, we have that
\begin{align}\label{eq: odd W}
W^{odd}(x)
&=\frac{1}{(2\pi)^{\frac{3}{2}}}\int_{-\infty}^{+\infty}\int_{-\pi}^{\pi}e^{\frac{1}{2}\ln(\frac{1}{2}\cosh(2\sqrt{2d\beta}yx+2h)+\frac{1}{2}\cos u)/x^2}e^{-y^2/2}\notag\\
&\quad \times \frac{e^{-\sqrt{2d\beta}yx-h}+e^{\sqrt{2d\beta}yx+h}\cos u}{\sqrt{2\cosh(2\sqrt{2d\beta}yx+h)+2\cos u}}\,\mathrm{d}y\mathrm{d}u\notag\\
&\xlongequal{t=\sqrt{2d\beta}yx+h}\frac{1}{(2\pi)^{\frac{3}{2}}\sqrt{2d\beta}|x|} \int_{-\infty}^{+\infty}\int_{-\pi}^{\pi}e^{\psi(t,u)/x^2}\rho(t,u)\,\mathrm{d}t\mathrm{d}u,
\end{align}
where $\psi(t,u)$ is defined by \eqref{eq: defn psi} and
\begin{equation*}
\rho(t,u)=\frac{e^{-t}+e^{t}\cos u}{\sqrt{2\cosh(2t)+2\cos u}}.
\end{equation*}
Then, we could find the asymptotic expansion of $W^{odd}(x)$ via Laplace's method. We  need to find the asymptotic expansion of the integrand $e^{\psi(t,u)/x^2}\rho(t,u)$. Compared with that of $W(x)$, the only difference is that we need to use the Taylor's expansion of $\rho(t,u)$ near the maximum point of $\psi(t,u)$.

Note that $\rho(t,u)\leq\rho(t,0)=1$.

For the critical case $h=0$ and $\beta=\beta_c$, the maximum point of $\psi(t,u)$ is $(0,0)$. It is also the maximum point of $\rho(t,u)$. Note that $\rho(t,u)$ is a even function in $u$ and
\[\rho(0,0)=1,\,\, \frac{\partial\rho}{\partial t}(0,0)=\frac{\partial^2\rho}{\partial t^2}(0,0)=\frac{\partial\rho}{\partial u}(0,0)=\frac{\partial^2\rho}{\partial t\partial u}(0,0)=0, \,\,\frac{\partial^2\rho}{\partial u^2}(0,0)\neq 0.\]
Hence, the Taylor expansion of $\rho(0,0)$ at point $(0,0)$ is
\begin{equation*}
    \rho(t,u)=1+\sum_{\substack{p+4q\geq 3, \\ p\geq 0, q\geq 0}}\widetilde{\alpha}_{p,2q}t^{p}u^{2q}.
\end{equation*}
By $t=s\sqrt{x}=O(x^{\frac{1}{2}})$ and $u=vx=O(x)$, we get
\[\rho(t,u)=1+O(|x|^{\frac{3}{2}}), \,\,\text{as}\,\,x\to 0^+.\]
Hence, we have that
\[W^{odd}(x)=W(x)(1+O(|x|^{\frac{3}{2}})), \,\,\text{as}\,\,x\to 0^+.\]
Hence, by \eqref{eq: O even via W} and \eqref{eq: O odd via W odd}, we get
\[O^{odd}_{d,\beta,h}=O_{d,\beta,h}(1+O(n^{-\frac{3}{4}})), \,\,\text{as}\,\, n\to+\infty.\]
Hence, by \eqref{eq: Qn = O/Z}, \eqref{eq: defn Pn} and Corollary~\ref{cor: Q asymptotics even}, we know that Proposition~\ref{prop: Pn expansion odd}(2) holds.

For the case $h=0$, $\beta<\beta_c$ or the case $h\neq 0$, $(t_*,0)$ is the only maximum point of $\psi(t,u)$. It is also the maximum point of $\rho(t,u)$. Near the point $(t_*,0)$, we have
\begin{equation*}\label{eq: rho taylor two cases}
    \rho(t,u)=\rho(t_*,0)+O(|t-t_{*}|^2+|u|^2)=1+O(|t-t_{*}|^2+|u|^2).
\end{equation*}
Intuitively, the main contribution of the integral in \eqref{eq: odd W} comes from the part when $t-t_{*}=sx=O(x)$ and $u=vx=O(x)$. Hence, we have $\rho(t,u)=1+O(x^2)$, as $x\to 0^+$. Therefore, we have that
\[W^{odd}(x)=W(x)(1+O(x^2)),\text{ as }x\to 0^+.\]
Hence, again by \eqref{eq: Qn = O/Z}, \eqref{eq: O even via W}, \eqref{eq: defn Pn}, \eqref{eq: O odd via W odd} and Corollary~\ref{cor: Q asymptotics even}, we know that Proposition~\ref{prop: Pn expansion odd}(1)(4) hold.

For the case $h=0$ and $\beta>\beta_c$, $\psi(t,u)$ has two different maximum points $(t_{*},0)$ and $(-t_*,0)$ where $t_*>0$. Near $(t_{*},0)$, we have
\[\rho(t,u)=1+O(|t-t_{*}|^2+|u|^2).\]
And near $(-t_{*},0)$, we have
\[\rho(t,u)=1+O(|t+t_{*}|^2+|u|^2).\]
Hence, by $t=\pm t_{*}+sx$ and $u=vx$, then
$\rho(t,u)=1+O(x^2)$, as $x\to 0^+$. Hence, we have that
\[W^{odd}(x)=W(x)(1+O(x^2)),\text{ as }x\to 0^+.\]
Hence, again by \eqref{eq: Qn = O/Z}, \eqref{eq: O even via W}, \eqref{eq: defn Pn}, \eqref{eq: O odd via W odd} and Corollary~\ref{cor: Q asymptotics even}, we know that Proposition~\ref{prop: Pn expansion odd} (3) holds.

And the above argument could be made rigorous.
\end{proof}

Hence, by \eqref{eq: Pn and Qn odd}, Corollary~\ref{cor: Q asymptotics even} and Proposition~\ref{prop: Pn expansion odd}, we prove Theorem~\ref{thm: Littlewood-Offord problem asymptotics} when $n$ is odd.

Hence, by Subsubsection~\ref{subsubsect: even n} and Subsubsection~\ref{subsubsection: odd n} we complete the proof of Theorem~\ref{thm: Littlewood-Offord problem asymptotics}.

\section{Acknowledgements}

The authors would like to thank Longmin Wang and Qiang Zeng for their discussions about the Curie-Weiss models. The first author was supported by National Natural Science Foundation of China \#11701395. The second author was supported by National Natural Science Foundation of China \#12001389.

\bibliographystyle{amsalpha}

\begin{thebibliography}{Woj06}

\bibitem[Erd45]{ErdosMR0014608}
P.~Erd\"{o}s, \emph{On a lemma of {L}ittlewood and {O}fford}, Bull. Amer. Math.
  Soc. \textbf{51} (1945), 898--902. \MR{14608}

\bibitem[Erd56]{ErdelyiMR0078494}
A.~Erd\'{e}lyi, \emph{Asymptotic expansions.}, Dover Publications, Inc., New
  York, 1956. \MR{78494}

\bibitem[Erd65]{ErdosMR0174539}
P.~Erd\H{o}s, \emph{Extremal problems in number theory.}, Proc. {S}ympos.
  {P}ure {M}ath., {V}ol. {VIII}, 1965, pp.~181--189. \MR{174539}

\bibitem[FS61]{FulksSatherMR0138945}
W.~Fulks and J.~O. Sather, \emph{Asymptotics. {II}. {L}aplace's method for
  multiple integrals}, Pacific J. Math. \textbf{11} (1961), 185--192.
  \MR{138945}

\bibitem[FV18]{FriedliVelenikMR3752129}
S.~Friedli and Y.~Velenik, \emph{Statistical mechanics of lattice systems},
  Cambridge University Press, Cambridge, 2018, A concrete mathematical
  introduction. \MR{3752129}

\bibitem[JK21]{JKMR4201801}
Tomas Ju\v{s}kevi\v{c}ius and Valentas Kurauskas, \emph{On the
  {L}ittlewood-{O}fford problem for arbitrary distributions}, Random Structures
  Algorithms \textbf{58} (2021), no.~2, 370--380. \MR{4201801}

\bibitem[Kle70]{KleitmanMR0265923}
Daniel~J. Kleitman, \emph{On a lemma of {L}ittlewood and {O}fford on the
  distributions of linear combinations of vectors}, Advances in Math.
  \textbf{5} (1970), 155--157. \MR{265923}

\bibitem[LO43]{LittlewoodOffordMR0009656}
J.~E. Littlewood and A.~C. Offord, \emph{On the number of real roots of a
  random algebraic equation. {III}.}, Rec. Math. [Mat. Sbornik] N.S. (1943),
  277--286. \MR{9656}

\bibitem[Ngu12]{NguyenMR2891377}
Hoi~H. Nguyen, \emph{A new approach to an old problem of {E}rd{\H{o}}s and
  {M}oser}, J. Combin. Theory Ser. A \textbf{119} (2012), no.~5, 977--993.
  \MR{2891377}

\bibitem[Rao21]{RaoMR4294326}
Shravas Rao, \emph{The {L}ittlewood-{O}fford problem for {M}arkov chains},
  Electron. Commun. Probab. \textbf{26} (2021), Paper No. 47, 11. \MR{4294326}

\bibitem[Sin22]{SinghalMR4440097}
Mihir Singhal, \emph{Erd\"{o}s-{L}ittlewood-{O}fford problem with arbitrary
  probabilities}, Discrete Math. \textbf{345} (2022), no.~11, Paper No. 113005,
  13. \MR{4440097}

\bibitem[Sta80]{StanleyMR0578321}
Richard~P. Stanley, \emph{Weyl groups, the hard {L}efschetz theorem, and the
  {S}perner property}, SIAM J. Algebraic Discrete Methods \textbf{1} (1980),
  no.~2, 168--184. \MR{578321}

\bibitem[SZ18]{ShamisZeitouniMR3824953}
Mira Shamis and Ofer Zeitouni, \emph{The {C}urie-{W}eiss model with complex
  temperature: phase transitions}, J. Stat. Phys. \textbf{172} (2018), no.~2,
  569--591. \MR{3824953}

\bibitem[TV06]{TaoVuMR2289012}
Terence Tao and Van Vu, \emph{Additive combinatorics}, Cambridge Studies in
  Advanced Mathematics, vol. 105, Cambridge University Press, Cambridge, 2006.
  \MR{2289012}

\bibitem[TV12]{TaoVuMR2965282}
\bysame, \emph{The {L}ittlewood-{O}fford problem in high dimensions and a
  conjecture of {F}rankl and {F}\"{u}redi}, Combinatorica \textbf{32} (2012),
  no.~3, 363--372. \MR{2965282}

\bibitem[Woj06]{WojdyloMR2219312}
John Wojdylo, \emph{Computing the coefficients in {L}aplace's method}, SIAM
  Rev. \textbf{48} (2006), no.~1, 76--96. \MR{2219312}

\end{thebibliography}

\providecommand{\bysame}{\leavevmode\hbox to3em{\hrulefill}\thinspace}
\providecommand{\MR}{\relax\ifhmode\unskip\space\fi MR }
\providecommand{\MRhref}[2]{%
  \href{http://www.ams.org/mathscinet-getitem?mr=#1}{#2}
}
\providecommand{\href}[2]{#2}

\end{document}